\documentclass{amsart}

\usepackage{amsmath,amssymb,amsthm}
\usepackage{hyperref}
\usepackage{xcolor}
\usepackage{bbm}
\usepackage{mathrsfs}

\newtheorem{thmintr}{Theorem}

\newtheorem{prop}{Proposition}[subsection]
\newtheorem{thm}[prop]{Theorem}
\newtheorem{lem}[prop]{Lemma}

\theoremstyle{definition}

\newtheorem{rem}[prop]{Remark}

\newtheorem*{ack}{Acknowledgement}

\def\co{\colon\thinspace}

\newcommand{\AAA}{\mathcal A}

\newcommand{\C}{\mathbb C}

\newcommand{\rmd}{\mathrm d}

\newcommand{\rme}{\mathrm e}

\newcommand{\HH}{\mathcal H}

\newcommand{\rmi}{\mathrm i}

\newcommand{\LLL}{\mathscr L}

\newcommand{\bfp}{\mathbf p}

\newcommand{\bfq}{\mathbf q}

\newcommand{\R}{\mathbb R}

\newcommand{\TT}{\mathcal T}

\newcommand{\VV}{\mathcal V}

\newcommand{\bfx}{\mathbf x}

\newcommand{\bfy}{\mathbf y}

\newcommand{\Z}{\mathbb Z}

\newcommand{\bfz}{\mathbf z}

\newcommand{\ra}{\rightarrow}

\DeclareMathOperator{\Hess}{\mathrm{Hess}}

\DeclareMathOperator{\id}{\mathrm{id}}

\DeclareMathOperator{\st}{\mathrm{st}}

\DeclareMathOperator{\supp}{\mathrm{supp}}

\begin{document}

\author{Jan Eyll}
\author{Jonas Fritsch}
\author{Kai Zehmisch}
\address{Fakult\"at f\"ur Mathematik, Ruhr-Universit\"at Bochum,
Universit\"atsstra{\ss}e 150, D-44801 Bochum, Germany}
\email{Jan.Eyll@rub.de}
\email{Jonas.Fritsch@rub.de}
\email{Kai.Zehmisch@rub.de}

\title[Contactomorphic vertically convex domains]
{Contactomorphic vertically convex domains}

\date{21.01.2026}

\begin{abstract}
  We consider the standard Darboux space
  equipped with the radial symmetric contact form.
  We study co-orientation preserving contactomorphisms
  between relatively compact domains up to the boundary.
  We determine the contactomorphism classes among all
  strict vertically convex domains
  over a round ball in the Liouville hyperplane
  that are radially symmetric about the Reeb axis
  and whose boundary coincide
  along a neighbourhood of the common equator.
  The total invariant is
  the mean curvature of the bounding sphere
  at the umbilic points with the same sign.
  Replacing the Liouville hyperplane
  by codisc bundles of closed non-Besse
  Riemannian manifolds
  or finite symplectisations of closed
  non-Besse strict contact manifolds
  analogous results are formulated
  in terms of characteristic length
  and total characteristic action, resp.
\end{abstract}

\thanks{This work is part of a project in the SFB/TRR 191
{\it Symplectic Structures in Geometry, Algebra and Dynamics},
funded by the DFG}

\maketitle



\section{Introduction\label{sec:intro}}

Let $V$ be a smooth manifold
and consider a connected, compact hypersurface $S$
in $\R\times V$.
Assume that $S$ bounds a (unique) relatively compact domain.
The closure $D_S$ is called the {\bf domain of} $S$.

We assume that $S$ is {\bf strict vertically convex} in the following sense:
There exist smooth functions $f_{\pm}\co V_S\ra\R$
defined on an open subset $V_S$ of $V$
with $f_-<f_+$.
Denote the graphs of $f_{\pm}$ by $S^{\pm}$.
Both functions $f_{\pm}$ are assumed to extend
over the boundary $\partial V_S$ of $V_S$ continuously
with $f_-=f_+$ along $\partial V_S$.
We require $S$ to be the (disjoint) union of
the {\bf entrance set} $S^-$, the {\bf exit set} $S^+$
and the graph of $f_{\pm}|_{\partial V_S}$,
the so-called {\bf equatorial set} $\partial S^{\pm}$ of $S$.
In case $\partial S^{\pm}$ is a submanifold of $S$
we will say {\bf equator} for short and write $E=\partial S^{\pm}$.
The triple $(S,V_S,f_{\pm})$ is called a {\bf shape}
and $V_S$ the {\bf shadow set} of $S$.

If $V$ admits an exact symplectic form $\rmd\lambda$,
a preferred choice of a Liouville primitive $\lambda$
allows to form its contactisation $(\R\times V,\rmd b+\lambda)$
together with the contact structure 
$\xi_{\lambda}=\ker(\rmd b+\lambda)$.
A strict vertically convex hypersurface $S$
in $(\R\times V,\rmd b+\lambda)$
can be used to provide a boundary condition
to compact contact manifolds with boundary.
If a collar neighbourhood is required to be
contactomorphic to a collar neighbourhood of $S$
in $(D_S,\xi_{\lambda})$,
the diffeomorphism type naturally compares to $D_S$.
This was used
in order to determine the diffeomorphism type
under additional assumptions on the supported Reeb dynamics,
see \cite{bschz19,beck23,eh94,grz14,gz16b,kschz20,kwz22}.
To come up with the contactomorphism type
known tools are the classification of tight contact structures
on the $3$-ball (\cite{elia92}) and on a solid torus (\cite{mak98})
in terms of the characteristic foliation on the boundary.

In order to find examples
for which the contactomorphism type
of the domain of a strict vertically convex hypersurface
can be determined, we will use the following notion:
Two strict vertically convex hypersurfaces
$S$ and $T$ in $\R\times V$
are called to admit {\bf equivalent}
shapes $(S,V_S,f_{\pm})$ and $(T,V_T,g_{\pm})$
whenever $V_S=V_T$ and the germs of
continuous functions $f_+$ and $g_+$
as well as $f_-$ and $g_-$, resp,
coincide along $\partial V_S=\partial V_T$.
In particular, $S$ and $T$ coincide in a neighbourhood
of the common equatorial set.

Consider the standard contactisation
$\big(\R\times\R^{2n},\rmd b+\tfrac12(\bfx\rmd\bfy-\bfy\rmd\bfx)\big)$,
whose contact structure is denoted by $\xi_{\st}$.
A function $f$ defined on an open subset of $\R^{2n}$
is called {\bf radially symmetric} provided that $f$
is invariant under the action of the orthogonal group.
We will identify $f(\bfx,\bfy)$ with $f(r)$ for $r=|(\bfx,\bfy)|$.

\begin{thmintr}
\label{thmintr:contactbodies}
Let $(S,V_S,f_{\pm})$ and $(T,V_T,g_{\pm})$
be equivalent shapes in $\R\times\R^{2n}$
with shadow set $V_S=V_T$
equal to a round ball about $0$ in $\R^{2n}$.
Assume that $f_{\pm}$ and $g_{\pm}$ are all radially symmetric. 
Then the domains $(D_S,\xi_{\st})$ and $(D_T,\xi_{\st})$
are co-orientation preserving contactomorphic up to the boundary
if and only if $f''_+(0)=g''_+(0)$ and $f''_-(0)=g''_-(0)$.
\end{thmintr}

In other words,
vertically convex domains
with equivalent radially symmetric shapes are contactomorphic
precisely if the mean curvatures at the umbilic points of the same sign
coincide.
The constructive part uses a Moser integrating vector field,
so that in fact
an ambient compactly supported
contact isotopy relative to
a neighbourhood of the equator is obtained.
The induced isotopy of the boundary of the vertically convex domain
is through radially symmetric shapes.

In Theorem \ref{thm:similar} and Section \ref{subsec:symmsteinshadow}
we will consider {\it symmetric} shapes
with shadow sets of more general Stein type.
We will describe a perturbation method
in order to obtain shapes with contactomorphic domains.

To formulate an example we consider
a closed, connected Riemannian manifold $Q$.
The total space $V=T^*Q$ of the cotangent bundle $\pi\co T^*Q\ra Q$
is provided with the canonical Liouville $1$-form
$\lambda=\lambda_Q$,
which at $u\in T^*Q$ is given by $u\circ T\pi$.
The contactisation 
$\big(\R\times V,\rmd b+\lambda\big)$
is the $1$-jet bundle of $Q$ with
contact structure $\ker(\rmd b+\lambda)$
denoted by $\xi_Q$.

We call a smooth function $f$ on an open subset of $T^*Q$
{\bf fibred radially symmetric} provided
$f(u)$ can be identified with $f(r)$ for $r=|u|$
w.r.t.\ the dual metric.
Define the {\bf characteristic half length}
\[
\LLL(f_{\pm}):=
\pm\int_0^{r_0}
          \frac{f'_{\pm}(\tau)}{\tau}
\rmd\tau
\]
whenever $f=f_{\pm}$
continuously extends to the closed co-disc bundle
$D_{r_0}T^*Q$ of radius $r_0>0$
and defines the entrance or exit set
of a strict vertically convex hypersurface $S$
in $\R\times V$ via the graph.
Belonging to the exit/entrance set
determines the sign $\pm$.
The quantity $\LLL(f_{\pm})$
is derived from the characteristic foliation
on the graph of $f_{\pm}$
in $\big(\R\times V,\rmd b+\lambda\big)$.
Namely $|\LLL(f_{\pm})|$
is the length of a geodesic arc
that is obtained from
the projection of a half characteristic to $Q$
by deleting consecutive equal length subarcs
of opposite orientation,
see Sections \ref{subsec:chrfol} and
\ref{subsec:charhalflength}.
Define the {\bf characteristic length} to be
\[
\Sigma(S):=
\LLL(f_+)+\LLL(f_-)
\,.
\]

\begin{thmintr}
\label{thmintr:cobocotangentconstr}
Let $(S,V_S,f_{\pm})$ and $(T,V_T,g_{\pm})$
be equivalent shapes in $\R\times T^*Q$
with shadow set $V_S=V_T$
equal to the interior of $D_{r_0}T^*Q$.
Assume that $f_{\pm}$ and $g_{\pm}$ are
fibred radially symmetric and that
$\Sigma(S)=\Sigma(T)$.

Then the domains $(D_S,\xi_Q)$ and $(D_T,\xi_Q)$
are co-orientation preserving contactomorphic
up to the boundary.
Moreover, the existing contactomorphism
can be assumed
  	\begin{enumerate}
 	\item[(a)]
	to restrict to $\id\times\,\eta_{\theta_0}$
	in a neighbourhood of the equator,
	where $\eta_{\theta_0}$ is
	the normalised cogeodesic flow
	at time $\theta_0$, and
 	\item[(b)]
	to map $\{f_{\pm}(0)\}\times Q$
	to $\{g_{\pm}(0)\}\times Q$ via
	$\big(f_{\pm}(0),q\big)\mapsto
	\big(g_{\pm}(0),q\big)$, resp.
  	\end{enumerate}
\end{thmintr}

Theorem \ref{thmintr:cobocotangentconstr}
is based on a Moser argument analogous to
the {\it if part} of
Theorem \ref{thmintr:contactbodies}
after potentially applying
a normalised cogeodesic twist
considered in Section \ref{subsec:normcogeotwists}.
Observe, that the reparametrisation $\eta_{\theta/r_0}$
is the Reeb flow at time $\theta/r_0$
on the co-sphere bundle $S_{r_0}T^*Q$
with respect to the contact form $\alpha^{r_0}$
that is obtained from restricting $\lambda_Q$
to $T(S_{r_0}T^*Q)$.
The proof of
Theorem \ref{thmintr:cobocotangentconstr}
shows that if the equality
$\LLL(f_{\pm})=\LLL(g_{\pm})$
holds for the characteristic half length, resp.,
the time $\theta_0$
in Theorem \ref{thmintr:cobocotangentconstr} (a)
can be chosen to be zero,
so that the resulting contactomorphism 
fixes the equator pointwise.

To formulate a converse statement
to Theorem \ref{thmintr:cobocotangentconstr}
notice that by Section \ref{subsec:equatorpreservation}
any contactomorphism $(D_S,\xi_Q)\ra(D_T,\xi_Q)$
can be assumed to preserve the equator
$E=\{f_+(r_0)\}\times S_{r_0}T^*Q$
up to contact isotopy.
Such an equator preserving contactomorphism
restricts to a contactomorphism on $E$
equipped with $\ker\alpha^{r_0}$.
In Theorem \ref{thmintr:cobocotangentobstr}
we will derive contact invariants of contactomorphisms
that restrict to {\it strict} contactomorphisms on $(E,\alpha^{r_0})$.

For that recall,
that $Q$ is called {\bf Besse}
if all geodesics of $Q$ are closed.
By a result of Wadsley \cite{wad75}
this is equivalent
to the periodicity of the geodesic flow of $Q$.
Notice,
that the geodesic flow is periodic if and only if
$\eta_P$ commutes with the multiplication by $-1$
on cotangent vectors for a nonzero real number $P$.
The minimal period of the normalised cogeodesic flow $\eta_{\theta}$
equals $2P$ in this case.
If $Q$ is not Besse we will simply set $P=0$
in the following theorem:

\begin{thmintr}
\label{thmintr:cobocotangentobstr}
Let $(S,V_S,f_{\pm})$ and $(T,V_T,g_{\pm})$
be equivalent shapes in $\R\times T^*Q$
with shadow set $V_S=V_T$
equal to the interior of $D_{r_0}T^*Q$.
Assume that $f_{\pm}$ and $g_{\pm}$
are fibred radially symmetric
and that there exists a co-orientation preserving
contactomorphism $\varphi$ of the domains
$(D_S,\xi_Q)$ and $(D_T,\xi_Q)$
up to the boundary.
Assume further that
$\varphi$ restricts to a strict contactomorphism
$\varepsilon$ on the equator $\{f_+(r_0)\}\times S_{r_0}T^*Q$
with respect to $\alpha^{r_0}$.
  	\begin{enumerate}
 	\item[(i)]
	Then there exists an integer $k$
	such that
	\[
	\Sigma(S)=\Sigma(T)+kP
	\,.
	\]
 	\item[(ii)]
	If $\varepsilon$	commutes with the fibre multiplication by $-1$,
	then there exist integers $k_{\pm}$ such that
	\[
	\LLL(f_{\pm})=\LLL(g_{\pm})+k_{\pm}P
	\,.
	\]
	Hence, we obtain $k=k_++k_-$ in (i).
 	\item[(iii)]
	If $\varphi$ restricts to a map
	$\big(g_{\pm}(0),\iota\big)$ from
	$\{f_{\pm}(0)\}\times Q$ to $\{g_{\pm}(0)\}\times Q$
	such that $\iota\circ\pi=\pi\circ\varepsilon$,  
	then there exist integers $k_{\pm}$ such that
	\[
	\LLL(f_{\pm})=\LLL(g_{\pm})+2k_{\pm}P
	\,.
	\]
	Hence, we obtain $k=2(k_++k_-)$ in (i).
  	\end{enumerate}
Here $2P$ is the minimal period
of the geodesic flow if $Q$ is Besse;
zero otherwise.
\end{thmintr}

In Section \ref{subsec:normcogeotwists}
we will construct contactomorphisms
between domains in $\R\times T^*Q$
that restrict to $\varepsilon=\eta_{\theta_0}$ on the equator,
send $\{f_{\pm}(0)\}\times Q$
to $\{g_{\pm}(0)\}\times Q$ via the identity map $\iota$
and satisfy $\LLL(f_{\pm})-\LLL(g_{\pm})=\pm\theta_0$
for any given value $\theta_0$.
In particular,
$\Sigma(S)=\Sigma(T)$ for all these examples
as claimed in Theorem \ref{thmintr:cobocotangentobstr} (i).
This shows that the assumption
in Theorem \ref{thmintr:cobocotangentobstr} (ii)
is essential.
Furthermore
taking $Q$ to be Besse with minimal period $2P$
and $\theta_0=kP$ shows that all integer multiples
in Theorem \ref{thmintr:cobocotangentobstr} (ii)
appear, see Remark \ref{rem:bessetwist}.
Observe that for $k$ even
$\eta_{kP}$ is the identity on the equator,
for $k$ odd it is not.

Involving Dehn--Seidel twists on $T^*S^n$
over the unit sphere $S^n\subset\R^{n+1}$,
whose geodesic flow is of period $2\pi$,
there exist contactomorphic domains in $\R\times T^*S^n$
as in Theorem \ref{thmintr:cobocotangentobstr} (ii)
with $\LLL(f_{\pm})-\LLL(g_{\pm})=\pm(k-j)\pi$
for any $j,k\in\Z$,
see Remark \ref{rem:bessetwist}.
The constructed contactomorphisms
reparametrise the equator via $\varepsilon=\eta_{k\pi}$
and map $\{f_{\pm}(0)\}\times S^n$ to $\{g_{\pm}(0)\}\times S^n$ 
via $\iota$ equal to $j$-times the antipodal map on $S^n$.
In particular, $\varepsilon$ is a bundle map over $\iota$
with respect to $\pi$
as in Theorem \ref{thmintr:cobocotangentobstr} (iii)
provided that $j\equiv k$ modulo $2$.

Interpolating the geometry of
Theorem \ref{thmintr:contactbodies}
and Theorem
\ref{thmintr:cobocotangentconstr} and \ref{thmintr:cobocotangentobstr}
we will consider shadow sets with disconnected boundaries,
cf.\ Section \ref{subsec:steintypesympl}:
Let $(M,\alpha)$ be a connected, closed
strict contact manifold,
which defines an exact symplectic manifold $(V,\lambda)$
via its symplectisation $(\R\times M,\rme^a\alpha)$.
Here $a\in\R$ is the coordinate in $\R$-direction.
The contactisation on $\R\times(\R\times M)$
has contact structure $\ker(\rmd b+\rme^a\alpha)$
denoted by $\xi_{\alpha}$.

A function $f$ defined
on $(a_1,a_0)\times M$ for real numbers $a_1<a_0$
is called $\R$-{\bf symmetric} provided that $f\equiv f(a)$.
Define the {\bf characteristic action}
\[
\AAA(f_{\pm}):=
\pm\int_{a_1}^{a_0}
\rme^{-\tau}
f'_{\pm}(\tau)\,
\rmd\tau
\]
whenever $f=f_{\pm}$
continuously extends to
$[a_1,a_0]\times M$
and defines the entrance or exit set
of a strict vertically convex hypersurface $S$
in $\R\times(\R\times M)$.
The quantity $\AAA(f_{\pm})$
is the action of the Reeb arc in $(M,\alpha)$
obtained from a characteristic connecting
the components of the equator
$E=\{a_1,a_0\}\times M$
on the graph of $f_{\pm}$
by projection,
see Section \ref{subsec:acandchar}.
Define the {\bf total characteristic action} to be
\[
\TT(S):=
\AAA(f_+)+\AAA(f_-)
\,.
\]

\begin{thmintr}
\label{thmintr:contactbodiessympl}
Let $(S,V_S,f_{\pm})$ and $(T,V_T,g_{\pm})$
be equivalent shapes in the contactisation
of the symplectisation of $(M,\alpha)$
with shadow set $V_S=V_T$
equal to $(a_1,a_0)\times M$ for real numbers $a_1<a_0$.
Assume that $f_{\pm}$ and $g_{\pm}$ are $\R$-symmetric.
  	\begin{enumerate}
 	\item[(a)]
	If $\TT(S)=\TT(T)$,
	then the domains $(D_S,\xi_{\alpha})$
	and $(D_T,\xi_{\alpha})$
	are co-orientation preserving contactomorphic
	up to the boundary.
	Moreover,
	the existing contactomorphism can be assumed
	to restrict to the identity
	in a neighbourhood of the equator
	$\{a_1,a_0\}\times M$.
 	\item[(b)]
	Assume that there exists a co-orientation preserving
	contactomorphism $\varphi$ of the domains
	$(D_S,\xi_{\alpha})$ and $(D_T,\xi_{\alpha})$
	up to the boundary.
  		\begin{enumerate}
 		\item[(i)]
		If $\varphi$ restricts to a strict contactomorphism
		on one of the components of the equator
		$\{a_1,a_0\}\times M$ with respect to $\alpha$,
		then there exists an integer $k$
		such that
		\[
		\TT(S)=\TT(T)+2kP
		\,.
		\]
 		\item[(ii)]
		If $\varphi$ preserves each component
		of the equator $\{a_1,a_0\}\times M$ such that
		the projection to $(M,\alpha)$ coincides
		with a single strict contactomorphism $\varepsilon$,
		then there exist integers $k_{\pm}$ such that
		\[
		\AAA(f_{\pm})=\AAA(g_{\pm})+2k_{\pm}P
		\,.
		\]
  		\end{enumerate}
	Here $2P$ is the minimal period
	of the Reeb flow on $(M,\alpha)$
	in case $(M,\alpha)$ is Besse;
	zero otherwise.
  	\end{enumerate}
\end{thmintr}

A strict contact manifold is called {\bf Besse}
if its Reeb flow is periodic.
An application of Reeb twists
as in Section \ref{subsec:reebtwists}
shows that the characteristic action $\AAA(f_{\pm})$
is not a contact invariant in general.
Even if the restriction of a contactomorphism
to the equator is a strict contactomorphism of the equator
the example in Section \ref{subsec:reebtwists} shows
that the assumption of being equal to $\varepsilon$
is absolutely essential.


\section{Radial symmetric functions\label{sec:radsymmfcts}}


\subsection{Morse lemma and consequences\label{subsec:mlandcons}}

Consider a smooth function $f\co\R\ra\R$.
By the fundamental theorem of calculus
we find a smooth function $f_1\co\R\ra\R$
such that $f(t)-f(0)=tf_1(t)$,
namely $f_1(t):=\int_0^1f'(st)\rmd s$.
Defining smooth functions $f_{k+1}:=(f_k)_1$ inductively,
we obtain $f(t)-f(0)=t^{k+1}f_{k+1}(t)$
whenever $f^{(l)}(0)=0$ for $l=1,\ldots,k$.
Taking $(k+1)$-fold derivative at $0$ yields
\[
f_{k+1}(0)=\frac{f^{(k+1)}(0)}{(k+1)!}
\,.
\]

\begin{lem}[Root lemma]
\label{lem:rootlem}
Let $f$ be an even function.
Then there exists a smooth function $\tilde{f}\co\R\ra\R$
such that $f(t)=\tilde{f}(t^2)$ for all $t\in\R$.
\end{lem}

\begin{proof}
By Whitney extension,
it suffices to find a smooth function $\tilde{f}$ on $[0,\infty)$,
as $t^2$ is always non-negative.
We can assume that $f(0)=0$.
Define
\[
\tilde{f}(t):=f(\sqrt{t})
\]
for $t\geq0$.
As $f$ is even, $\tilde{f}(t^2)=f(|t|)=f(t)$ for all $t\in\R$.

It remains to show smoothness of $\tilde{f}$.
For $t>0$ we have by the Morse lemma that
\[
\tilde{f}'(t)
=\frac{f'(\sqrt{t})}{2\sqrt{t}}
=\frac12(f')_1(\sqrt{t})
\]
as $f'(0)=0$ for an even function.
Similarly, using $f(0)=0$,
\[
\frac{\tilde{f}(t)}{t}
=\frac{(\sqrt{t})^2f_2(\sqrt{t})}{t}
=f_2(\sqrt{t})
\]
for all $t>0$,
so that the left handed limit for $t$ tends to $0$
exists and equals
\[
\tilde{f}'(0)=f_2(0)=\frac12f''(0)=\frac12(f')_1(0)
=\lim_{t\ra0}\tilde{f}'(t)
\,.
\]
The second equality follows with the Morse lemma,
the third from letting $t$ tend to $0$ in $f'(t)/t=(f')_1(t)$,
and the last from the equation $\tilde{f}'(t)=\frac12(f')_1(\sqrt{t})$
obtained above.
In other word, $\tilde{f}'$ extends (from the left) to $0$ continuously,
so that $\tilde{f}$ is $C^1$ on $[0,\infty)$ with
\[
\tilde{f}'=\frac12\widetilde{(f')_1}
\,.
\]
Repeating the present argument with the smooth, even function
$(f')_1-(f')_1(0)$ instead of $f$ yields
that $\widetilde{(f')_1}$ is $C^1$ and, consequently,
that $\tilde{f}$ is $C^2$.
Inductively, this shows that $\tilde{f}$ is smooth.
\end{proof}


\subsection{About the Hessian\label{subsec:hess}}

The root lemma implies
that a smooth function $f$ on $\R^n$
is radially symmetric if and only if
there exists a smooth function $\tilde{f}\co\R\ra\R$
such that $f(\bfx)=\tilde{f}(r^2)$, $r=|\bfx|$,
for all $\bfx$,
see Lemma \ref{lem:rootlem}.
Therefore,
\[
\partial_jf(\bfx)
=2x_j\tilde{f}'(r^2)
\]
and, hence,
\[
\partial_{ij}f(\bfx)
=4x_ix_j\tilde{f}''(r^2)
+2\delta_{ij}\tilde{f}'(r^2)
\,.
\]
In the proof of the root lemma we obtained
$\tilde{f}'(0)=\frac12f''(0)$,
so that
\[
\Hess_f(0)=f''(0)\mathbb{I}
\,.
\]


\section{Invariants of contactomorphisms\label{sec:invofcontactom}}


\subsection{Characteristic foliation\label{subsec:chrfol}}

Given a co-oriented contact manifold $(M,\xi)$,
the contact structure $\xi$
can be written as the kernel of a contact form $\alpha$.
The defining contact form $\alpha$ is determined uniquely
up to multiplication by a positive function on $M$.
The co-orientation of $\xi$ equals the one induced by the Reeb vector field
of any contact form that defines $\xi$.
Analogously,
$\xi$ defines an orientation on $M$ via $\alpha\wedge(\rmd\alpha)^n$
and a conformal symplectic structure on $\xi$ via $\rmd\alpha$,
where the dimension of $M$ is $2n+1$.

Consider an oriented hypersurface $S$ of $M$
provided with a positive volume form $\Omega$.
Setting $\beta:=\alpha|_{TS}$,
the equation $i_X\Omega=\beta\wedge(\rmd\beta)^{n-1}$
defines a smooth vector field $X$ on $S$ uniquely.
Changing $\Omega$ or $\alpha$ by a positive function, resp.,
results in a conformal change of $X$ by a positive function.
In other words,
the conformal class $[X]$ only depends on the orientation of $S$
and on the co-oriented contact structure $\xi$.
The oriented singular $1$-dimensional foliation $S_{\xi}$
integrated by any representing {\bf characteristic vector field} $X$
is called {\bf characteristic foliation},
regular $1$-dimensional leaves are the {\bf characteristics}.
Observe that the set of characteristic vector fields is convex
and, therefore, allows cut and paste operations.

For any characteristic vector field $X$ of $S_{\xi}$
and for all $p\in S$ 
the linear span $\R X_p$
is the symplectic orthogonal complement of
$T_pS\cap\xi_p=\ker\beta_p$
in $\xi_p$
w.r.t.\ the conformal symplectic structure on $\xi$,
see \cite[Lemma 2.5.20]{gei08}.
The {\bf singular points} of $S_{\xi}$ are those $p\in S$
for which $T_pS=\xi_p$ holds.
A singular point $p$ of $S_{\xi}$ is called {\bf positive}
whenever the orientations of $T_pS$ and $\xi_p$ coincide,
{\bf negative} otherwise.
After taking appropriate choices,
the set of singular points,
the so-called {\bf singular set},
is equal to the zero set
of any characteristic vector field of $S_{\xi}$.
Alternatively,
the singular set can be given by $\{p\in S\,|\,\beta_p=0\}$.
Hence,
by Frobenius non-integrability of $\xi$,
the singular set has no interior point.
Its complement, the so-called {\bf regular set}, is open.

The linearisation of a characteristic vector field $X$ of $S_{\xi}$
w.r.t.\ any connection $D$ on $S$
defines a linear endomorphism $DX|_p$ on $T_pS$
for all $p\in S$.
At a singular point $p\in S$ the endomorphism $DX|_p$
is independent of the choice of a connection
as it is given by the directional derivatives in a coordinate patch.
Furthermore, by product rule,
different choices of a characteristic vector field $X$
result in a change of $DX|_p$ by a positive constant.
In particular, the singularity $p$ of $S_{\xi}$
determines the eigenvalues of $DX|_p$
up to multiplication by a positive constant.

The previous discussion implies:

\begin{prop}
\label{prop:phiinvonchrfol}
Let $\varphi$ be a co-orientation preserving contactomorphism of $(M,\xi)$.
Consider an oriented hypersurface $T$ of $M$.
Assume that $\varphi$ restricts to an orientation preserving
diffeomorphism from $S$ onto $T$.
Then $\varphi$ preserves the oriented characteristic foliations,
restricts to a diffeomorphism between the regular sets
sending characteristics to characteristics and
induces a bijection between the signed singularities
preserving their eigenvalues conformally.
\end{prop}

\begin{proof}
The pull back action of $\varphi$ on $\alpha$
is given by multiplication with a positive function on $M$.
The pull back of a positive volume form on $T$ by $\psi:=\varphi|_S$
is a positive volume form on $S$.
Therefore,
given a characteristic vector field $X$ of $S_{\xi}$,
the push forward $\psi_*X$
will be a characteristic vector field of $T_{\xi}$ by the very definition.
This immediately implies the proposition
possibly
up to the statement about the eigenvalues,
which can obtained as follows:
The covariant derivative of $\psi_*X$
w.r.t.\ a connection $D'$ on $T$ is obtained from $DX$
by conjugation with $T\psi$,
where $D=\psi^*D'$ is the pull back connection on $S$.
\end{proof}


\subsection{Hamiltonian representation\label{subsec:hamrepr}}

Let $\alpha$ be a $\xi$-defining contact form on $M$
and denote its Reeb vector field by $R_{\alpha}$.
Consider a smooth function $H$ on $M$
and its {\bf contact Hamiltonian vector field} $X_H$,
which is the unique solution of
\[
\alpha(X_H)=H
\quad\text{and}\quad
i_{X_H}\rmd\alpha=
-\rmd H+\rmd H(R_{\alpha})\alpha
\,.
\]
Observe that $R_{\alpha}=X_1$.

\begin{prop}
\label{prop:hamvectischrvect}
Assume that $S$ is the regular zero set of $H$
and that $H$ growths in normal direction of $S$
according the co-orientation.
Assume further that $S_{\xi}$ admits singular points.
Then $X_H|_S$ generates the oriented singular $1$-dimensional foliation $S_{\xi}$.
\end{prop}

\begin{proof}
The choice of $H$ implies that
$\pm\rmd H(R_{\alpha})|_p>0$
at a positive/negative singularity $p$ of $S_{\xi}$
because $\pm T_pS$ equals $\xi_p$
as oriented vector spaces.
By the defining equations of the contact Hamiltonian $X_H$,
the restriction $X=X_H|_S$ is tangent to $\xi$ and $S$.
Moreover,
$X$ vanishes precisely at the singular points of $S_{\xi}$.
Hence, point-wise,
$\R X$ is the symplectic orthogonal complement of $TS\cap\xi$
by the very definition.

To conclude that $X$ is a characteristic vector field for $S_{\xi}$
observe that
\[
i_{X_H}(\rmd\alpha)^n=
n\big(\!-\!\rmd H+\rmd H(R_{\alpha})\alpha\big)
\wedge(\rmd\alpha)^{n-1}
\,.
\]
Restriction to $TS$ yields
\[
i_X(\rmd\beta)^n=
n\cdot\rmd H(R_{\alpha})\beta\wedge(\rmd\beta)^{n-1}
\]
setting $\beta=\alpha|_{TS}$.
As $\rmd\beta$ is symplectic near the positive/negative singularity $p$
inducing the right/wrong orientation of $S$ near $p$,
\[
\frac
{(\rmd\beta)^n}
{n\cdot\rmd H(R_{\alpha})}
\]
is a positive volume form of $S$ near $p$.
Therefore,
$X$ generates $S_{\xi}$ with the correct orientation.
\end{proof}

Consequently,
different choices of $\alpha$ and $H$
result in a multiplication of $X_H|_S$ by a positive function on $S$.


\subsection{Contactisation\label{subsec:contactisation}}

Consider the contactisation $(M,\alpha)=(\R\times V,\rmd b+\lambda)$
and assume that a hypersurface $S$ is the graph
of a smooth function $f$ on $V$.
In other words,
$S=\{\pm H=0\}$ is the regular zero set
of the Hamiltonian $H$ given by $H(b,v)=b-f(v)$ for all $v\in V$.
The co-orientation of $S$ is induced by $\pm\partial_b$.

Observe $R_{\alpha}=\partial_b$
and that $\rmd\alpha=\rmd\lambda$
restricts to a symplectic form on $S$.
In particular,
with the arguments made for Proposition \ref{prop:hamvectischrvect},
the contact Hamiltonian vector field $\pm X_{H}=X_{\pm H}$
restricts to a characteristic vector field on $S=\{\pm H=0\}$.
Denoting by $Y_{\lambda}$ the $\rmd\lambda$-dual of $\lambda$
and by $X_f$ the Hamiltonian vector field of $f$,
i.e.\ the unique solution of
\[
i_{Y_{\lambda}}\rmd\lambda=\lambda
\quad\text{and}\quad
i_{X_f}\rmd\lambda=-\rmd f
\,,
\]
resp.,
the contact Hamiltonian $X_H$ of $H=b-f$ is computed to be
\[
X_H=\big((b-f)+\lambda(X_f)\big)\partial_b+Y_{\lambda}-X_f
\,.
\]
Hence,
the induced characteristic vector field on $S$
equals
\[
X=\pm\big(\lambda(X_f)\partial_b+Y_{\lambda}-X_f\big)
\,.
\]


\subsection{Radial symmetry\label{subsec:radsymm}}

Let $V$ be an open neighbourhood of the origin in $\R^{2n}$
and take $\lambda=\tfrac12(\bfx\rmd\bfy-\bfy\rmd\bfx)$,
so that $Y_{\lambda}=\tfrac12(\bfx\partial_{\bfx}+\bfy\partial_{\bfy})$.
Assume for simplicity that $f(0)=0$
and that $0$ is a critical point of $f$.
Hence,
$T_0S=\{0\}\times\R^{2n}$ in $\R\times\R^{2n}$
and
$X_f=-J\nabla f$ is singular at $0$,
denoting the standard complex structure on $\R^{2n}$ by $J$.
 
Define a connection $D$ on $S$ 
by taking directional derivate on $\R\times\R^{2n}$
followed by orthogonal projection onto $TS$.
Then,
\[
DX|_0=\pm\big(\tfrac12\mathbb{I}+J\Hess_f(0)\big)
\,.
\]
If $f$ is radially symmetric we obtain
together with Section \ref{subsec:hess}
\[
DX|_0=\pm\big(\tfrac12\mathbb{I}+f''(0)J\big)
\,.
\]
Notice,
that the eigenvalues of $\tfrac12\mathbb{I}+f''(0)J$
equal $\frac12\pm f''(0)\rmi$.
As $Y_{\lambda}$ is perpendicular to the regular level sets of $f$,
which are round spheres about the origin,
$0$ is the unique singularity of $S_{\xi}$.

\begin{proof}[{\bf Proof of Theorem \ref{thmintr:contactbodies} -- only if part}]
Let $\varphi$ be an co-orientation preserving contactomorphism
from $(D_S,\xi_{\st})$ onto $(D_T,\xi_{\st})$.
Then, by Proposition \ref{prop:phiinvonchrfol},
$\varphi$ maps the signed singularity
$\pm(f_{\pm}(0),0)$ to $\pm(g_{\pm}(0),0)$ and
the eigenvalues $\pm\big(\frac12\pm f_{\pm}''(0)\rmi\big)$
and $\pm\big(\frac12\pm g_{\pm}''(0)\rmi\big)$
coincide up to a positive constant.
Comparing the real parts this constant is $1$ necessarily; 
comparing the imaginary parts afterwards yields
$|f_{\pm}''(0)|=|g_{\pm}''(0)|$.
Furthermore
the linearisation of $\varphi$ at the singularities
conjungates $\pm f_{\pm}''(0)J$ to $\pm g_{\pm}''(0)J$
(as the positive constant is observed to be equal to $1$)
and, therefore, maps
$\pm f_{\pm}''(0)\cdot\rmd\lambda(\,.\,,J\,.\,)$
to
$\pm g_{\pm}''(0)\cdot\rmd\lambda(\,.\,,J\,.\,)$
conformally.
Hence, 
$f_{\pm}''(0)=g_{\pm}''(0)$ as claimed.
\end{proof}


\subsection{Preserving the equator\label{subsec:equatorpreservation}}

Consider a shape $(S,V_S,f_{\pm})$
in $(\R\times V,\rmd b+\lambda)$
and assume that the closure of the shadow set $V_S$
is a Liouville domain.
In particular, $V_S$ has smooth boundary
and $Y_{\lambda}$ is transverse to $\partial V_S$ 
pointing outward.
Moreover,
assume that $f_{\pm}$ is constant along $\partial V_S$,
equal to zero, say.
Hence, the equator
$E=\{0\}\times\partial V_S=\pm\partial S^{\pm}$
is a smooth submanifold with contact form $\lambda|_{TE}$.

Choose a tubular neighbourhood $U_{\delta}$ of
$[-\delta,\delta]_b\times\partial V_S$
such that the restriction of $Y_{\lambda}$ to $U_{\delta}$
is transverse to $S$.
Let $\chi\equiv\chi(b)$ smooth function on $U_{\delta}$
equal to $\pm1$ for $b$ between $\pm\delta$ and $\pm\infty$
and strictly increasing on $(-\delta,\delta)$
with $\chi(b)=b$ near zero.
Extend $\chi$ to a smooth function on $\R\times V$
such that $\chi$ is $\pm1$ near $S^{\pm}\setminus U_{\delta}$.
The contact Hamiltonian vector field $X_{\chi}$ of $\chi$
is transverse to $\partial D_S$ pointing outwards
being equal to $\chi\partial_b+\chi'Y_{\lambda}$ along $S$.
Hence,
$S$ is a {\bf Giroux-convex} hypersurface.
As $X_{\chi}$ is annulated by $\rmd b+\lambda$
precisely along $E$,
the {\bf dividing set} defined by $X_{\chi}$ equals $E$.

Therefore,
given equivalent shapes
and a contactomorphism between their domains
we can assume after a contact isotopy
that moves the dividing sets onto each other
that the equator is preserved,
cf.\ \cite[Theorem 4.8.5(b)]{gei08}.
Such an equator preserving contactomorphism
restricts to a contactomorphism on $E$
equipped with the contact structure $\ker\lambda|_{TE}$.

\begin{rem}
\label{rem:divset}
 As discussed in \cite[Section 4.8]{gei08},
 on a closed Giroux-convex hypersurface $S$
 of $(M,\xi)$ there exists a smooth function $u\co S\ra\R$
 with the following properties:
 The zero locus $\Gamma$ of $u$ is non-empty
 and equal to a dividing set.
 Furthermore
 $X(u)<0$ along $\Gamma$
 for any characteristic vector field $X$ of $S_{\xi}$
 so that $\Gamma$ is the smooth oriented boundary
 of $\{u>0\}$ positively transvers to $S_{\xi}$.
 
  Observe that by \cite[Example 4.8.10]{gei08}
  there exists Giroux-convex hypersurfaces
  whose characteristic foliation
  does not admit any singularities.
 We will come back to this in
 Remark \ref{rem:nongirouxconvex}.
\end{rem}


\section{Integrating a contactomorphism\label{sec:intgratcont}}

We will prove the {\it if part} of Theorem \ref{thmintr:contactbodies}.
Throughout this section
we consider equivalent shapes
$(S,V_S,f_{\pm})$ and $(T,V_T,g_{\pm})$
in a contactisation $(\R\times V,\rmd b+\lambda)$.
The domains in $\R\times V$
bounded by the strict vertically convex hypersurfaces $S$ and $T$
are denoted by $D_S$ and $D_T$, resp.


\subsection{Vertically diffeomorphic\label{subsec:vertdiff}}

A diffeomorphism $\phi$ of $\R\times V$ is called {\bf vertical}
provided there exists a smooth function $f\co\R\times V\ra\R$
such that $\partial_bf>0$ and
$\phi(b,v)=\big(f(b,v),v\big)$ for all $(b,v)\in\R\times V$.

\begin{lem}
\label{lem:verticallydiffeomorphic}
  There exist open neighbourhoods $U^{\pm}\!$ of
  $S^{\pm}$ in $\R\times V_S$, resp.,
  and a vertical diffeomorphism $\phi=(f,\id)$ of
  $\;\R\times V$
  such that
  	\begin{enumerate}
 	\item[(i)]
  		$\phi(D_S)=D_T$,
 	\item[(ii)]
		$f(b,v)=b+g_{\pm}(v)-f_{\pm}(v)$
		for all $(b,v)\in U^{\pm}$, and
 	\item[(iii)]
		$\phi$ is the time-$1$ map of a vertical isotopy
		with compact support in $\R\times V_S$,
		outside of which $f$ will be equal to $(b,v)\mapsto b$.
  	\end{enumerate}
\end{lem}

\begin{proof}
  Let $\varepsilon>0$ be small and consider the functions
  \[
  \max\{f_-,g_-\}+\varepsilon
  \quad\text{and}\quad
  \min\{f_-,g_-\}-\varepsilon
  \]
  on $V_S$.
  Let $A_1^-$ be the set of points in $\R\times V_S$
  that lie between its graphs.
  Let $A_0^-$ be the set of points in $\R\times V_S$
  that lie above or below the graph of
  \[
  \max\{f_-,g_-\}+2\varepsilon
  \quad\text{and}\quad
  \min\{f_-,g_-\}-2\varepsilon
  \]
  resp.
  Choose a smooth function $\chi_-\co\R\times V_S\ra [0,1]$
  that is $1$ on $A_1^-$ and $0$ on $A_0^-$.
  Define a vector field
  \[
  \chi_-(b,v)
  \Big(g_-(v)-f_-(v)\Big)
  \cdot\partial_b
  \]
  on $\R\times V$ thinking of $g_--f_-$ being smoothly extended
  to $V$ by $0$.
  The time-$1$ map of the flow
  yields a vertical diffeomorphism $\phi_-$
  with compact support in $\R\times V_S$,
  which maps $S^-$ onto $T^-$ and restricts to
  \[
  (b,v)\longmapsto\Big(b+g_-(v)-f_-(v),v\Big)
  \]
  for all $(b,v)\in A_1^-$.
  
  The image $\phi_-\big(\{b=f_+\}\big)$ of $S^+$ is the graph of
  $\tilde{f}_+=f_++(g_--f_-)$.
  Repeating the construction with $f_-,g_-$ replaced by $\tilde{f}_+,g_+$
  yields a vertical diffeomorphism $\phi_+$
  with compact support in $\R\times V_S$,
  which maps $\phi_-(S^+)$ onto $T^+$ and restricts to
  \[
  (b,v)\longmapsto\Big(b+g_+(v)-\tilde{f}_+(v),v\Big)
  \]
  for all $(b,v)\in A_1^+$.
  Choosing $\varepsilon>0$ small enough
  ensures $\phi_+=\id$ in a neighbourhood of $\phi_-(S^-)=T^-$
  as $g_+-\tilde{f}_+$ vanishes near $\partial V_S$.
  Then $\phi=\phi_+\circ\phi_-$
  is the desired vertical diffeomorphism.
\end{proof}

\begin{rem}[Odd-symplectic uniqueness]
\label{rem:oddsympuniness}
Pulling back $\rmd b+\lambda$
along a vertical diffeomorphism $\phi=(f,\id)$
yields the contact form $\rmd f+\lambda$,
which coincides with $\rmd b+\lambda$ in the complement
of a compact subset of $\R\times V_S$.
Taking the exterior derivative,
Lemma \ref{lem:verticallydiffeomorphic} implies that
$(D_S,\rmd\lambda)$ and $(D_T,\rmd\lambda)$
are vertically odd-symplectomorphic.
\end{rem}


\subsection{Interpolation}
\label{subsec:interpolation}

Lemma \ref{lem:verticallydiffeomorphic} yields a
vertical isotopy $\phi_t$ of $\R\times V$
of the form $\phi_t(b,v)=\big(f_t(b,v),v\big)$
for all $(t;b,v)\in[0,1]\times\R\times V$
with compact support in $\R\times V_S$.
Here $f_t\co\R\times V\ra\R$ is a smooth family of functions
such that $\partial_bf_t>0$.
Define a family of positive contact forms
$\rmd f_t+\lambda$ by pulling back
$\rmd b+\lambda$ along $\phi_t$.

As the space of vertical diffeomorphisms is convex,
$\phi_t$ is isotopic to the vertical isotopy
for which $f_t$ is replaced by the convex combination
of $f$ and $(b,v)\mapsto b$.
Therefore, we consider the convex homotopy
\[
\alpha_t:=(1-t)\rmd b+t\,\rmd f+\lambda
\]
of compactly supported contact forms.


\subsection{Moser problem}
\label{subsec:moserproblem}

Consider an isotopy $\varphi_t$ of $\R\times V$.
Taking time-derivative $\dot{\varphi_t}=X_t|_{\varphi_t}$
we obtain a time-dependent vector field $X_t$ on $\R\times V$.
Assuming $\varphi_t$ has compact support,
$X_t$ is defined uniquely.
Picard--Lindel\"of yields the converse.
Therefore, the isotopy $\varphi_t$,
whose support we assume to be compact
and contained in $\R\times V_S$,
can be identified with $X_t$ uniquely.
Observe, that $\supp(\varphi_t)=\supp(X_t)$.

In addition,
let $f_t$ be a smooth family of positive functions on $\R\times V_S$
such that the support of $f_t-1$ is contained in $\supp(\varphi_t)$.
As explained in \cite{gei08},
the pair $(\varphi_t,f_t)$ solves $\varphi_t^*\alpha_t=f_t\alpha_0$
for all $t\in[0,1]$
(so that $\varphi_0=\id$ implies $f_0=1$)
if and only if
the pair $(X_t,g_t)$ consisting of the $\varphi_t$-generating
vector field $X_t$ and $g_t:=(\ln f_t)\,\dot{}\circ\varphi_t^{-1}$
solves
\[
\dot\alpha_t+\rmd H_t+i_{X_t}\rmd\alpha_t=g_t\alpha_t
\]
for all $t\in[0,1]$,
where $H_t:=\alpha_t(X_t)$.
Observe, that $f_t=\rme^{\int_0^tg_s\circ\varphi_s\rmd s}$
and the support of $g_t$ is contained in $\supp(\varphi_t)$.

\begin{rem}
\label{rem:frommosertocontactom}
In order to prove
the {\it if part} of Theorem \ref{thmintr:contactbodies}
it suffices to find $(X_t,g_t)$
with compact support in $\R\times V_S$
such that $\varphi_t(S)\subset S$
for all $t\in[0,1]$.
Indeed, in view of Lemma \ref{lem:verticallydiffeomorphic},
$\phi\circ\varphi_1$ will be a contactomorphism
from $(D_S,\xi_{\lambda})$ onto $(D_T,\xi_{\lambda})$.
\end{rem}

Therefore,
it will be more than enough to find $(X_t,g_t)$ such that
$X_t$ is tangent to $S$ for all $t\in[0,1]$.
This means
\[
\rmd f_{\pm}(X_t)|_{S^{\pm}}=\rmd b(X_t)|_{S^{\pm}}
\]
for all $t\in[0,1]$
as equivalence of shapes allows $X_t=0$
near the equatorial set.
Observe that $\varphi_t(S^{\pm})=S^{\pm}$
for all $t\in[0,1]$ afterall.


\subsection{Decomposition}
\label{subsec:decomposition}

Abbreviate
\[
\eta_t:=t\,\rmd_V\!f
\]
denoting the horizontal part of $\rmd$ by $\rmd_V\!$.
Decompose
\[
\alpha_t=\frac{1}{r_t}\rmd b+
\eta_t+\lambda
\,,
\]
where the coefficient of the vertical part
\[
\frac{1}{r_t}:=(1-t)+t\partial_bf
\]
is a positive function.
Finally,
writing $s:=\partial_bf-1$ and $\delta:=\rmd_V\!f$,
we have
\[
\dot\alpha_t=
s\rmd b+\delta
\,.
\]
Observe that the given data
$\eta_t$, $r_t-1$, $s$ and $\delta$
have compact support in $\R\times V_S$.
Restricted to the neighbourhoods $U^{\pm}\!$ of $S^{\pm}$
we have $\rmd_V\!f=\rmd(g_{\pm}-f_{\pm})$,
$r_t=1$ and $s=0$,
see (ii) in Lemma \ref{lem:verticallydiffeomorphic}.

Similarly,
we decompose the integrating vector field
$X_t=H_tR_{\alpha_t}+Y_t$.
The Reeb vector field of $\alpha_t$
is equal to $R_{\alpha_t}=r_t\partial_b$
as $\rmd\alpha_t=\rmd\lambda$.
The component tangent to the kernel of $\alpha_t$
can be written as $Y_t=y_t\partial_b+V_t$
with $V_t$ being horizontal.
Hence,
\[
X_t=(H_tr_t+y_t)\partial_b+V_t
\,.
\]

\begin{lem}
\label{lem:decompmosprob}
  The Moser problem
  formulated in Section \ref{subsec:moserproblem}
  admits a solution $(\varphi_t,f_t)$
  if and only if
  there exist $H_t$, $y_t$, $V_t$ and $g_t$ such that
  	\begin{enumerate}
 	\item[(i)]
	  $y_t=-r_t(\eta_t+\lambda)(V_t)$,
 	\item[(ii)]
	  $i_{V_t}\rmd\lambda=-\delta-\rmd_V\!H_t+g_t(\eta_t+\lambda)$,
 	\item[(iii)]
	  $0=s+\partial_bH_t-\frac{g_t}{r_t}$,
  	\end{enumerate}
 and the boundary condition
  	\begin{enumerate}
 	\item[(iv)]
	  $
	  H_t|_{S^{\pm}}=
	  \big((1-t)\rmd f_{\pm}+t\rmd g_{\pm}+\lambda\big)(V_t)|_{S^{\pm}}
	  $
  	\end{enumerate}
  are satisfied.
  Furthermore $(\varphi_t,f_t)$ will have
  compact support in $\R\times V_S$
  precisely if the $H_t$, $y_t$, $V_t$ and $g_t$ do.
\end{lem}

\begin{proof}
  The kernel condition $\alpha_t(Y_t)=0$ is equivalent to (i).
  The Lie condition from Section \ref{subsec:moserproblem}
  decomposes into (ii) and (iii) equivalently.
  The boundary condition from Section \ref{subsec:moserproblem}
  equals (iv), in which (i) is plugged in
  using $\rmd_V\!f=\rmd(g_{\pm}-f_{\pm})$ and $r_t=1$
  on the neighbourhoods $U^{\pm}$ of $S^{\pm}$
  in $\R\times V_S$,
  see Lemma \ref{lem:verticallydiffeomorphic}.
  The statement about the supports is immediate from (i), (ii) and (iii).
\end{proof}


\subsection{Achieving the boundary condition}
\label{subsec:achthebc}

We will continue the use 
of horizontal Liouville and Hamiltonian vector fields
as introduced in Section \ref{subsec:contactisation}.

\begin{lem}
\label{lem:intvecfield}
  Let $H_t$ be a smooth family of functions
  on $\R\times V$ with compact support
  contained in $U^+\cup\,U^-$,
  see Lemma \ref{lem:verticallydiffeomorphic}.
  Assume that the family of horizontal vector fields
  \[
    V_t=X_f+X_{H_t}+r_t(s+\partial_bH_t)(Y_{\lambda}-tX_f)
  \]
  satisfy the boundary condition (iv) in Lemma \ref{lem:decompmosprob}.
  Then the Moser problem formulated in Section \ref{subsec:moserproblem}
  admits a compactly supported solution $(\varphi_t,f_t)$.
\end{lem}

\begin{proof}
  Define $y_t$ by (i) in Lemma \ref{lem:decompmosprob}
  using the given $V_t$.
  For given $H_t$,
  formula (iii) determines $g_t=r_t(s+\partial_bH_t)$.
  Plugged in into the right hand side of (ii) yields
  the one of
  \[
    i_{V_t}\rmd\lambda=
    -\delta-\rmd_V\!H_t+r_t(s+\partial_bH_t)(\eta_t+\lambda)
    \,,
  \]
  so that the given $V_t$ is the $\rmd\lambda$-dual horizontally,
  i.e.\ (ii) in Lemma \ref{lem:decompmosprob} holds.
  As $V_t$ admits compact support in $\R\times V_S$
  so $y_t$ and $g_t$ do.
  The claim follows with Lemma \ref{lem:decompmosprob}.
\end{proof}

Let $H_t$ be a solution as in Lemma \ref{lem:intvecfield}.
We will utilise that, by (ii) in Lemma \ref{lem:verticallydiffeomorphic},
$\rmd_V\!f=\rmd(g_{\pm}-f_{\pm})$, $r_t=1$ and $s=0$ holds
on the neighbourhoods $U^{\pm}\!$ of $S^{\pm}$.
Hence
  \[
    V_t|_{U^{\pm}}=
    X_f+
    X_{H_t}+
    \partial_bH_t\cdot(Y_{\lambda}-tX_f)
    \,.
  \]
The boundary condition (iv) in Lemma \ref{lem:decompmosprob}
can be phrased like
  \[
    H_t|_{S^{\pm}}=
    \big(t\rmd_V\!f+\lambda\big)(V_t)|_{S^{\pm}}+
    \rmd f_{\pm}(V_t)|_{S^{\pm}}
    \,.
  \]
In order to compute the right hand side observe that
  \[
    t\rmd_V\!f
    \big(X_f+X_{H_t}\big)=
    -\rmd_V\!H_t\big(tX_f\big)
  \]
by horizontal energy preservation and that
  \[
    \lambda
    \big(X_f+X_{H_t}\big)=
    \lambda(X_f)+
    \rmd_V\!H_t(Y_{\lambda})
    \,.
  \]
Adding both using $\rmd\lambda$-duality as
  \[
    \big(t\rmd_V\!f+\lambda\big)
    \big(\partial_bH_t\cdot(Y_{\lambda}-tX_f)\big)
    =0
  \]
yields
  \[
    \big(t\rmd_V\!f+\lambda\big)(V_t)=
    \lambda(X_f)+\rmd_V\!H_t\big(Y_{\lambda}-tX_f\big)
    \,.
  \]
Furthermore
  \[
    \rmd f_{\pm}(V_t)=
    \rmd f_{\pm}(X_f)-
    \rmd_V\!H_t\big(X_{f_{\pm}}\big)+
    \partial_bH_t\cdot\rmd f_{\pm}\big(Y_{\lambda}-tX_f\big)
    \,,
  \]
so that the desired right hand side
  \[
    \big(t\rmd_V\!f+\lambda\big)(V_t)+
    \rmd f_{\pm}(V_t)
  \]
is equal to
  \[
    \big(\rmd f_{\pm}+\lambda\big)(X_f)+
    \big(\partial_bH_t\cdot\rmd f_{\pm}+\rmd_V\!H_t\big)
    \big(Y_{\lambda}-tX_f-X_{f_{\pm}}\big)
  \]
as $\rmd f_{\pm}$ annihilates $X_{f_{\pm}}$.
Similarly, as
$\big(\rmd f_{\pm}+\lambda\big)(X_f)=
\rmd_V\!f\big(Y_{\lambda}-X_{f_{\pm}}\big)$,
we get
  \[
    \big(\rmd_V\!f+\partial_bH_t\cdot\rmd f_{\pm}+\rmd_V\!H_t\big)
    \big(Y_{\lambda}-tX_f-X_{f_{\pm}}\big)
    \,.
  \]
In order to evaluate along $S^+=\{b=f_+\}$ we set
  \[
    h_t^{\pm}(v):=H_t\big(f_{\pm}(v),v\big)
  \]
for all $v\in V_S$.
By chain rule
  \[
    \rmd h_t^{\pm}=\partial_bH_t\cdot\rmd f_{\pm}+\rmd_V\!H_t
    \,,
  \]
so that the boundary condition becomes
  \[
    h_t^{\pm}=
    \big(\rmd_V\!f+\rmd h_t^{\pm}\big)
    \big(Y_{\lambda}-tX_f-X_{f_{\pm}}\big)
  \]
on $V_S$ finally.

The converse is also true.

\begin{prop}
\label{prop:charisteqn}
  Let $h_t^{\pm}$ be a smooth family of functions
  on $V_S$ with compact support that solves
  \[
    h_t^{\pm}=
    \rmd\big(h_t^{\pm}+(g_{\pm}-f_{\pm})\big)
    \big(Y_{\lambda}-X_{f_{\pm}+t(g_{\pm}-f_{\pm})}\big)
  \]
  for each of the signs $\pm$ individually.
  Then $(D_S,\xi_{\lambda})$ and $(D_T,\xi_{\lambda})$
  are contactisotopic relative a neighbourhood
  of the common equatorial set.
\end{prop}

\begin{proof}
  Let $U^{\pm}_1$ be small neighbourhoods of $S^{\pm}$ in $V_S$
  whose closures in $V_S$ are contained in $U^{\pm}\!$, resp.
  For all $(b,v)\in U^{\pm}_1$ define $H_t(b,v)=h_t^{\pm}(v)$
  and extend $H_t$ to all of $\R\times V$ via cuting off
  inside $U^{\pm}\!$.
  The resulting $H_t$ is a smooth family of functions
  on $\R\times V$ with compact support
  contained in $U^+\cup\,U^-$
  as required in Lemma \ref{lem:intvecfield}.
  Choosing $U^{\pm}_1$ small enough ensures
  that the two $\pm$-branches of $H_t$ do not overlap.
  Using $\rmd_V\!f=\rmd(g_{\pm}-f_{\pm})$ restricted to $U^{\pm}\!$
  and reading the above computation
  of the right hand side of the boundary condition backwards,
  one obtains that $V_t$ solves
  the boundary condition (iv) in Lemma \ref{lem:decompmosprob}.
  With Lemma \ref{lem:intvecfield} we get
  that the Moser problem formulated in Section \ref{subsec:moserproblem}
  admits a compactly supported solution $(\varphi_t,f_t)$.
  The claim follows now as in Remark \ref{rem:frommosertocontactom}.
\end{proof}

\begin{rem}
\label{rem:singsetshift}
  The restriction of the solution obtained above to $S^{\pm}$
  satisfies $\partial_bH_t=0$,
  which results into
  $V_t|_{S^{\pm}}=X_{g_{\pm}-f_{\pm}}+X_{H_t}$.
\end{rem}


\subsection{Method of characteristics}
\label{subsec:characteristics}

For each $t\in[0,1]$ denote by
$\psi_s\equiv\psi_s^t$ the local flow of the Liouville vector field
$Y_{\lambda}-X_{f_{\pm}+t(g_{\pm}-f_{\pm})}$
treating each of the signs $\pm$ individually.
In view of a solution $h\equiv h_t^{\pm}$
of the characteristic equation for the boundary condition
in Proposition \ref{prop:charisteqn}
we make the following remarks:

Evaluating
  \[
    \psi_s^*h=
    \Big(\psi_s^*\big(h+(g_{\pm}-f_{\pm})\big)\Big)'
  \]
at $s=0$, where $(\ldots)'$ stands for $\frac{\rmd}{\rmd s}$,
yields
  \[
    h=
    L_{
      (Y_{\lambda}-X_{f_{\pm}+t(g_{\pm}-f_{\pm})})
      }
    \big(h+(g_{\pm}-f_{\pm})\big)
    \,,
  \]
which is the characteristic equation for the boundary condition
in Proposition \ref{prop:charisteqn}.
Observe that at equilibrium points a solution $h$ vanishes.
Further, variation of the constant yields
  \[
    \psi_s^*h=
    \rme^s
      \left(
        h-
        \int_0^s
          \rme^{-\sigma}
          \big(
            \psi_{\sigma}^*(g_{\pm}-f_{\pm})
          \big)'
        \rmd\sigma
    \right)
    \,,
  \]
which by chain rule leads back to the above characteristic equation.
In other words,
$h$ is a solution if and only if
along each flow line $c(s)\equiv c_t^{\pm}(s)=\psi_s(v)$
for each initial point $v\in V_S$ we have
  \[
    h\big(c(s)\big)=
    \rme^s
      \left(
        h(v)-
        \int_0^s
          \rme^{-\sigma}
          \rmd(g_{\pm}-f_{\pm})
          \big(c'(\sigma)\big)\,
        \rmd\sigma
    \right)
    \,.
  \]
The integral along any flow line $c$
contained in the open subset
\[
U_0:=
V_S\setminus\supp\!\big(\rmd(g_{\pm}-f_{\pm})\big)
\]
of $V_S$ vanishes.
Consequently,
$h$ growths exponentially in forward time $s$
along each flow line $c$ in $U_0$.

In order to achieve compact support for $h$
we set $h(v)=0$ for all $v\in U_{\partial}$,
where
\[
U_{\partial}:=
V_S\setminus\supp\,(g_{\pm}-f_{\pm})
\,.
\]
Notice, that the closure of $U_{\partial}$
is a neighbourhood of $\partial V_S$ and
that the set of points in $V_S$
through which there passes a flow line $c$
entirely contained in $\supp\,(g_{\pm}-f_{\pm})$
contains no interior point.
This follows from exponential $\rmd\lambda$-volume growth
along the Liouville flow $\psi_s$ in forward time $s$.


\subsection{Liouville shadow sets}
\label{subsec:liouvilleshadow}

We continue the discussion from Section \ref{subsec:characteristics}.
The shape $(S,V_S,f_{\pm})$
in $(\R\times V,\rmd b+\lambda)$
is of {\bf Liouville domain type} if
  	\begin{enumerate}
 	\item[(a)]
	  $Y_{\lambda}$ is positively transverse to
	  (the {\it a posteriori} smooth boundary) $\partial V_S$,
 	\item[(b)]
	  $f_{\pm}$ restricted to $\partial V_S$ is locally constant and
 	\item[(c)]
	  $\rmd f_{\pm}(Y_{\lambda})$ restricted to a neighbourhood
	  of $\partial V_S$ is negative/positive.
	\end{enumerate}
Item (a) means,
that the closure of the shadow set $V_S$ equipped with
$Y_{\lambda}$ is a Liouville domain.
Furthermore, by (b) and (c),
the Liouville vector field $Y_{\lambda}-X_{f_{\pm}}$ on $V_S$
turns $V_{f_{\pm}}\!$ into a Liouville domain,
where $V_{f_{\pm}}\!$ is obtained from $V_S$ by removing
a small open collar neighbourhood
that consists entirely of level sets of $f_{\pm}$.

Consequently,
as for any Liouville domain,
the local flow $\psi_s^0$ of $Y_{\lambda}-X_{f_{\pm}}$
is defined for all $s\leq0$,
has non-empty, compact skeleton
\[
\Big\{
v\in V_{f_{\pm}}
\,\Big|\,
\psi_s^0(v)\in V_{f_{\pm}}
\,\,
\forall s\geq0
\Big\}
=
\bigcap_{s\leq0}\psi_s^0(V_{f_{\pm}})
\]
without interior points.
Moreover,
the map $(s,w)\mapsto\psi_s^0(w)$
is a diffeomorphism of $(-\infty,0]\times\partial V_{f_{\pm}}\!$
onto the complement of the skeleton in $V_{f_{\pm}}\!$.

The equivalent shapes
$(S,V_S,f_{\pm})$ and $(T,V_T,g_{\pm})$
are {\bf similar} if
  	\begin{enumerate}
 	\item[(x)]
	$(S,V_S,f_{\pm})$ is of Liouville domain type,
 	\item[(y)]
	the skeleta of the local flows
	$\psi_s\equiv\psi_s^t$ of
        $Y_{\lambda}-X_{f_{\pm}+t(g_{\pm}-f_{\pm})}$
	for all $t\in[0,1]$ are equal and
 	\item[(z)]
	$g_{\pm}-f_{\pm}$ is locally constant
	in a neighbourhood of the skeleton
	of $Y_{\lambda}-X_{f_{\pm}}$.
	\end{enumerate}
In view of (x) and the equivalence of shapes,
$(T,V_T,g_{\pm})$ is of Liouville domain type as well.
The requirement that the skeleta of all
$Y_{\lambda}-X_{f_{\pm}+t(g_{\pm}-f_{\pm})}$
have a common neighbourhood
on which $g_{\pm}-f_{\pm}$ is locally constant
is equivalent to (y) and (z).

\begin{thm}
\label{thm:similar}
  Assume that the shapes
  $(S,V_S,f_{\pm})$ and $(T,V_T,g_{\pm})$
  are similar.
  Moreover, assume that 
  for all $w\in\partial V_{f_{\pm}}\!$
  and for all $t\in[0,1]$
  \[
        \int_{-\infty}^0
          \rme^{-\sigma}
          \rmd(g_{\pm}-f_{\pm})
          \big(c'(\sigma)\big)\,
        \rmd\sigma
        =0
    \,,
  \]
  where $c(s)\equiv c_t^{\pm}(s)=\psi_s^t(w)$
  is the flow line of
  $Y_{\lambda}-X_{f_{\pm}+t(g_{\pm}-f_{\pm})}$
  starting at $w$.
  Then $(D_S,\xi_{\lambda})$ and $(D_T,\xi_{\lambda})$
  are contactisotopic relative a neighbourhood
  of the common equatorial set.
\end{thm}

\begin{proof}
  The aim is to find a solution $h\equiv h_t^{\pm}$
  of the problem formulated in
  Proposition \ref{prop:charisteqn}.
  By Section \ref{subsec:characteristics}
  it suffices to find a smooth $h$
  solving the formulated integral equation
  along all flow lines $c$.
  We set $h=0$ on $U_0$ in $V_S$,
  so that $h(w)=0$ for all $w\in\partial V_{f_{\pm}}\!$
  and $h=0$ on the skeleton by (y) and (z).
  By exponential growth
  there is no other choice on the skeleton.

  The complement of the skeleton in $V_{f_{\pm}}\!$
  is diffeomorphic to $(-\infty,0]\times\partial V_{f_{\pm}}\!$
  via $(s,w)\mapsto\psi_s^t(w)$.
  Define $h(v)$ for all $v\in V_{f_{\pm}}\!$
  not contained in the skeleton by
  \[
    h\big(c(s)\big)=
    \rme^s
        \int_s^0
          \rme^{-\sigma}
          \rmd(g_{\pm}-f_{\pm})
          \big(c'(\sigma)\big)\,
        \rmd\sigma
    \,,
  \]
  where $v=c(s)$ depends on the parameters
  $(s,w,t)\in(-\infty,0]\times\partial V_{f_{\pm}}\!\times[0,1]$
  smoothly and uniquely (separately for each sign $\pm$).
  The integral exists for $s\ra-\infty$ by (y) and (z)
  as it is constant for $|s|$ sufficiently large.
  Hence,
  by the integral assumption,
  $h=0$ in a neighbourhood of the skeleton
  as it needs to be.
\end{proof}


\subsection{Symmetric Stein type shadow sets}
\label{subsec:symmsteinshadow}

We present an example of the situation in
Section \ref{subsec:liouvilleshadow}:
  Assume that $(V,\lambda)$ admits a
  {\bf strictly plurisubharmonic potential},
  i.e.\ a smooth proper map $\varrho\co V\ra[0,\infty)$
  such that $\lambda=-\rmd\varrho\circ J$
  for a $\rmd\lambda$-compatible
  almost complex structure $J$ on $V$,
  see Section \ref{subsec:tamegeom} for an example.
  Hence,
  $Y_{\lambda}=\nabla\!\varrho$
  and $X_{\varrho}=J\nabla\!\varrho$,
  where the gradient is taken w.r.t.\
  the induced metric $\rmd\lambda(\,.\,,J\,.\,)$.
  
  This allows the following simplification:
  Consider smooth functions $f$ on an open sublevel set
  $\{\varrho<\varrho_*\}$ of $V$,
  where $\varrho_*$ is a positive constant
  not equal to a critical value of $\varrho$.
  Assume that $f$ factors smoothly through $\varrho$,
  so that $f(v)$ can be written as $\hat{f}\big(\varrho(v)\big)$
  for all $v\in\{\varrho<\varrho_*\}$
  and a smooth function $\hat{f}\co\R\ra\R$.
  Then {\it for any $g=\hat{g}(\varrho)$
  we can read $g(c)$ as taken along a flow line $c$
  of either $Y_{\lambda}-X_f$ or $Y_{\lambda}$
  equivalently}.
  To see this observe that
  by chain rule $X_f=\hat{f}'(\varrho)X_{\varrho}$.
  Hence,
  $Y_{\lambda}-X_f=(\mathbb{I}-\hat{f}'(\varrho)J)\nabla\!\varrho$,
  so that $Y_{\lambda}-X_f$ and $Y_{\lambda}$
  both vanish precisely at the critical points of $\varrho$.
  In the complement of the critical points of $\varrho$
  we use flow boxes over level sets of $\varrho$
  taken w.r.t.\ $X=\nabla\!\varrho/|\nabla\!\varrho|^2$.
  In there,
  $\varrho$ is given by a linear function
  in flow time with slope $1$;
  level sets of $\varrho$ are the orthogonal slices.
  The claim follows
  since $Y_{\lambda}$ is positively proportional to $X$
  while $X_f$ is orthogonal.
  
  We remark that the skeleton of $Y_{\lambda}-X_f$
  consists of all stable manifolds of $Y_{\lambda}-X_f$
  of the critical points of $\varrho$ in $V_S$
  and $\rmd\varrho(Y_{\lambda}-X_f)=|\nabla\!\varrho|^2>0$
  on $\{\rmd\varrho\neq0\}$.
  The skeleton of $Y_{\lambda}-X_f$ is not changed
  by adding to $f=\hat{f}(\varrho)$
  a smooth perturbation function $\chi=\hat{\chi}(\varrho)$
  with the following properties:
  $\rmd\chi$ has compact support in the interior of 
  the maximal collar neighbourhood of $\{\varrho\leq\varrho_*\}$
  obtained w.r.t.\ to $X$
  and vanishes near $\{\varrho=\varrho_*\}$.
  
  A regular compact hypersurface $S$ in $\R\times V$
  is determined by the set solutions $(b,v)$
  of the equation $b^2+\varrho(v)=\varrho_*$.
  The shape structure is induced by the functions
  $f_{\pm}(v)=\pm\sqrt{\varrho_*-\varrho(v)}$
  defined for all $v$ in $V_S=\{\varrho<\varrho_*\}$.
  Consider perturbation functions $\chi_{\pm}=\hat{\chi}_{\pm}(\varrho)$
  with
  \[
        \int_{-\infty}^0
          \rme^{-\sigma}
          \rmd\chi_{\pm}
          \big(c'(\sigma)\big)\,
        \rmd\sigma
        =0
        \,,
  \]
  where
  \[
  \rmd\chi_{\pm}\big(c'(\sigma)\big)
  =\Big(\chi_{\pm}\big(c(\sigma)\big)\Big)'
  \]
  can be read with flow line $c$ taken w.r.t.\
  $Y_{\lambda}-X_{f_{\pm}+t(g_{\pm}-f_{\pm})}$
  or $Y_{\lambda}$ equivalently.
  {\it The perturbation $g_{\pm}:=f_{\pm}+\chi_{\pm}$
  yields a shape $\big(T,V_S,f_{\pm}+\chi_{\pm}\big)$
  with domain contactisotopic to the one of $(S,V_S,f_{\pm})$}.
  This follows with Theorem \ref{thm:similar},
  in which the resulting function
  \[
    h\big(c(s)\big)=
    \rme^s
        \int_s^0
          \rme^{-\sigma}
          \rmd\chi_{\pm}
          \big(c'(\sigma)\big)\,
        \rmd\sigma
  \]
  is of type $h=\hat{h}(\varrho)$ also.

\begin{proof}
[{\bf Proof of Theorem \ref{thmintr:contactbodies}
-- if part}]
Identify $\C^n$ with $\R^{2n}$ by setting
$z_j=x_j+\rmi y_j$ for all $j=1,\ldots,n$ and recall
that the standard Stein structure on $\C^n$
is defined through the strictly plurisubharmonic function
$\varrho(\bfz)=\tfrac14|\bfz|^2$.
Observe that the radial Liouville primitive of $\rmd\bfx\wedge\rmd\bfy$
is given by
$-\rmd\varrho\circ\rmi=\tfrac12(\bfx\rmd\bfy-\bfy\rmd\bfx)$.

Consider the radially symmetric function
$g:=g_{\pm}-f_{\pm}$ defined on $\R^{2n}$
that has support in the open round ball
$B_{r_0}\!(0)$.
Identify $g(\bfz)=g(r)$, where $r=|\bfz|$,
and observe that by the root lemma
$g=\hat{g}(\varrho)$ factors through a smooth function
$\hat{g}=\tilde{g}\circ2\co\R\ra\R$
as $r\mapsto g(r)$ is even,
see Lemma \ref{lem:rootlem}.
Therefore,
as argued in Section \ref{subsec:symmsteinshadow},
we can work with the flow
$\psi_s(\bfz)=\rme^{s/2}\bfz$ of
the autonomous radial Liouville vector field
$\tfrac12(\bfx\partial_{\bfx}+\bfy\partial_{\bfy})$
and no Hamiltonian perturbations are necessary.

For $\bfz_0\in\partial B_{r_0}\!(0)$
denote the initiated flow line by
$c(s)=\rme^{s/2}\bfz_0$ for all $s<0$.
For all non-zero $\bfz\in  B_{r_0}\!(0)$
write $\bfz=\rme^{s/2}\bfz_0$ and define
$h(\bfz)=h(r)$ by
  \[
    h(\bfz)=
    \rme^s
        \int_s^0
          \rme^{-\sigma}
          \Big(g\big(r_0\rme^{\sigma/2}\big)\Big)'
          \,
        \rmd\sigma
        \,.
  \]
  With chain rule the integrand becomes
  \[
  \rme^{-\sigma}
  \cdot
  g'\big(r_0\rme^{\sigma/2}\big)
  \cdot
  \frac{1}{2}r_0
  \rme^{\sigma/2}
  \,,
  \]
  so that the substitution $\tau=r_0\rme^{\sigma/2}$,
  which gives $\rmd\tau=\frac12\tau\rmd\sigma$
  and $r=r_0\rme^{s/2}$, finally leads to
  \[
    h(\bfz)=
    r^2
        \int_r^{r_0}
          \frac{g'(\tau)}{\tau^2}
          \,
        \rmd\tau
        \,.
  \]
  
  In order to show smoothness at $\bfz=0$
  observe that $(g')'(0)=0$ by assumption.
  The Morse lemma implies
  that the integrant is equal to $(g')_2(\tau)$,
  see Section \ref{subsec:mlandcons}.
  Because $(g')_2$ is an odd smooth function on $\R$,
  the function given by
  $G(r)=\int_0^r(g')_2(\tau)\rmd\tau$
  is even and smooth on $\R$.
  By the root lemma we find a smooth function
  $\tilde{G}$ on $\R$
  such that $G(r)=\tilde{G}(r^2)$,
  see Lemma \ref{lem:rootlem}.
  Therefore,
  $h(\bfz)=|\bfz|^2\big(G(r_0)-\tilde{G}(|\bfz|^2)\big)$
  is smooth in $\bfz=0$ also.
  As in the proof of Theorem \ref{thm:similar}
  the claim follows with 
  Proposition \ref{prop:charisteqn}
  and Section \ref{subsec:characteristics}.
\end{proof}


\section{Cotangent shadows\label{sec:cotangshadow}}

We address Theorem
\ref{thmintr:cobocotangentconstr} and \ref{thmintr:cobocotangentobstr}.


\subsection{Tame geometry\label{subsec:tamegeom}}

As explained in \cite[Section 5]{kwz22}
the Liouville form $\lambda\equiv\lambda_Q$
and the Riemannian metric of $Q$
induce a tame structure on $T^*Q$:
The dual connection of the Levi-Civita connection on $Q$
defines a splitting of $TT^*Q=\HH\oplus\VV$
into horizontal $\HH$ and vertical distribution $\VV$.
An almost complex structure $J\equiv J_{T^*Q}$
is defined by sending horizontal/vertical
to $-$vertical/$+$horizontal vectors.
This yields the Sasaki metric
$\rmd\lambda(\,.\,,J\,.\,)$
on $T^*Q$ turning the splitting into an orthogonal one
and making $J$ compatible with the symplectic form $\rmd\lambda$.
In addition
\[
\lambda=
-\rmd\varrho\circ J\,,
\]
where we denote the kinetic energy function
by $\varrho(u):=\frac12|u|^2$ for $u\in T^*Q$.
Hence,
we are in the situation of
Section \ref{subsec:symmsteinshadow}.
In particular,
taking the gradient w.r.t.\ $\rmd\lambda(\,.\,,J\,.\,)$
yields that the Liouville vector field of $\lambda$
and the Hamiltonian vector field of $\varrho$
are equal to $Y_{\lambda}=\nabla\!\varrho$
and $X_{\varrho}=J\nabla\!\varrho$, resp.
Here, as in Section \ref{subsec:contactisation},
$Y_{\lambda}$ denotes the fibrewise radial
Liouville vector field on $T^*Q$
defined to be the $\rmd\lambda$-dual
of the canonical Liouville form $\lambda$ on $T^*Q$.

Notice that
\[
\lambda(X_{\varrho})=
Y_{\lambda}(\varrho)=
2\varrho
\]
by $2$-homogeneity of $\varrho$
in fibrewise radial direction.


\subsection{Fibred radial symmetry\label{subsec:fibwradsymm}}

Consider a shape $(S,V_S,f_{\pm})$ in $\R\times T^*Q$
with shadow set $V_S$ equal to the interior of $D_{r_0}T^*Q$.
Assume that both functions $f_{\pm}$ are fibred radially symmetric.
Write $f$ for one of the functions $f_{\pm}$.

Being fibred radially symmetric
allows the identification of $f(u)$ with $f(r)$ for $r=|u|$
by restricting to a line in any fibre of $T^*Q$.
The resulting function $\R\ra\R$ is even,
so that the root Lemma \ref{lem:rootlem} yields
a smooth function $\tilde{f}\co\R\ra\R$
with $f(u)=\tilde{f}(r^2)$ for all $u\in D_{r_0}T^*Q$
and $r=|u|$.
Setting $\hat{f}=\tilde{f}\circ2$ we obtain $f=\hat{f}(\varrho)$.
Restricting to the interior of $D_{r_0}T^*Q$,
chain rule yields
\[
X_f=
\hat{f}'(\varrho)X_{\varrho}
\]
for the respective Hamiltonian vector fields.
Notice, that the prefactor tends to $\mp\infty$ for $r\ra r_0$,
where the sign is determined by
whether $f$ defines the exit or the entrance set.

The projection of the characteristic vector field $X$ on $S$
computed in Section \ref{subsec:contactisation}
to the interior of $D_{r_0}T^*Q$
is the vector field
$Y_{\lambda}-X_f=(\mathbb{I}-\hat{f}'(\varrho)J)\nabla\!\varrho$
up to sign $\pm$.
The zero set is equal to the zero section $Q$ of $T^*Q$.
Therefore,
the singular set of the characteristic foliation $S_{\xi_Q}$
is given by $\{f(0)\}\times Q$,
along which the contact structure $\xi_Q$
is horizontal, equal to $\{f(0)\}\times T(T^*Q)|_Q$,
see Section \ref{subsec:chrfol}.
The sign is positive/negative provided that $f$ corresponds to 
the exit/entrance set of $S$,
see Section \ref{subsec:hamrepr}.
The equator $E$ of $S$ is $\{f(r_0)\}\times S_{r_0}T^*Q$.
Along $E$ the contact structure $\xi_Q$ is nowhere horizontal
as $\lambda$ is trivial precisely on the zero section.
Therefore,
all characteristics of $S_{\xi_Q}$
connect the positive singular set with the negative one
intersecting the equator transversally
{\it a posteriori} in a unique point.
The closures of the two complements of the equator
are called {\bf half characteristic}. 
The projections to $T^*Q$ along the $\R$-factor,
which can be parametrised by the flow lines
of the characteristic vector field $\pm(Y_{\lambda}-X_f)$
along the interior of $D_{r_0}T^*Q$,
connect each cosphere bundle $S_rT^*Q$, $r\in(0,r_0)$,
with the boundary $S_{r_0}T^*Q$
via arcs of finite length.


\subsection{Half characteristics dynamic\label{subsec:halfchardyn}}

The flow of the Liouville vector field $Y_{\lambda}$ on $T^*Q$
is given by $(s,u)\mapsto\rme^su$.
Recall,
$X_{\varrho}$ integrates the cogeodesic flow $\gamma_s$
on $T^*Q$,
cf.\ \cite[p.~27-31]{gei08}.
Both flows can be used to express
the flow $\psi_s$ of the vector field $Y_{\lambda}-X_f$
on the interior of $D_{r_0}T^*Q$.
We claim that
\[
\psi_s(u)=\rme^s\gamma_{\sigma_u(s)}(u)
\]
for all $u\in T^*Q$, $|u|<r_0$, and $s\in\R$, where
\[
\sigma_u(s):=
\int_s^0
  \rme^{\sigma}\hat{f}'\big(\varrho(\rme^{\sigma}u)\big)
\rmd\sigma\,.
\]
Indeed,
an application of
$i_{[X,Y]}=L_Xi_Y-i_YL_X$, see \cite[Proposition V.5.3]{lng02},
to $\rmd\lambda$
using $\lambda(X_{\varrho})=Y_{\lambda}(\varrho)=2\varrho$
yields the $\rmd\lambda$-dual of the equation
$[Y_{\lambda},X_{\varrho}]=X_{\varrho}$,
which in turn shows that the pull back of $X_{\varrho}$
along the flow of $-Y_{\lambda}$ equals $\rme^{-s}X_{\varrho}$.
Pretending for the moment
that $\psi_s$ just abbreviates
the right hand side of the claimed equation
we get that $\psi_0=\id$ and
$
\frac{\rmd}{\rmd s}\psi_s=
Y_{\lambda}|_{\psi_s}+
\sigma'_{(\,.\,)}(s)\cdot\rme^{-s}X_{\varrho}|_{\psi_s}
$.
As the cogeodesic flow $\gamma_{\sigma}$
preserves the kinetic energy $\varrho$
and, hence,
$\varrho(\rme^{\sigma}\,.\,)=
\varrho(\psi_s)$ holds,
computing $\sigma'$ finally shows
that $\psi_s$ is indeed the flow of $Y_{\lambda}-X_f$.

Observe that
$\rme^a\gamma_{\sigma}\circ\rme^{-a}=\gamma_{\rme^{-a}\sigma}$
holds for all $a\in\R$,
as the conjugated flow of $\gamma_{\sigma}$
is the flow of $\rme^{-a}X_{\varrho}$.
Therefore,
we can write
$\rme^{s-a}\gamma_{\rme^{-a}\sigma_u(s)}(\rme^au)$
for $\psi_s(u)$,
so that with $\rme^a=1/|u|$ we get
\[
\psi_s(u)=\rme^s|u|\cdot\gamma_{|u|\sigma_u(s)}
\left(\frac{u}{|u|}\right)
\]
for all $u\neq0$ with $|u|<r_0$.


\subsection{Oriented length\label{subsec:orientedlength}}

For given $u\in T^*Q$ of colength $|u|\in(0,r_0)$
consider the curve $s\mapsto\pi\big(\psi_s(u)\big)$ in $Q$
obtained by foot point projection.
By Section \ref{subsec:halfchardyn}
the curve reparametrises a geodesic arc in $Q$,
to which the metric dual of $u$ is tangent to.
Here,
we allow the reparametrisation to go forward and backward in time.
We will compute the oriented length $\ell_u(s)$
for time parameters between $0$ and $s$,
where the forward/backward oriented arcs
contribute positively/negatively.
With the arguments of \cite[Section 3.4]{zeh13}
the action
\[
a_u(s)=\int_s^0U^*\lambda
\]
of the lift $U(s)=\gamma_{\sigma_u(s)}(u)$
to $S_{|u|}T^*Q$ 
is the area $a_u(s)=|u|\ell_u(s)$
of the surface obtained by convex interpolation
with $s\mapsto\pi\big(\psi_s(u)\big)$.

The curve $s\mapsto\pi\big(\psi_s(u)\big)$
alternatively can be obtained from
\[
s\longmapsto
\pi\circ\gamma_{|u|\sigma_u(s)}
\left(\frac{u}{|u|}\right)
\,.
\]
Fixing $u$ and $s$ the oriented length of
the rescaled curve
\[
[0,1]\ni\tau\longmapsto
\pi\circ\gamma_{|u|\sigma_u(s\cdot\tau)}
\left(\frac{u}{|u|}\right)
\]
is equal to $\ell_u(s)$.
Removing consecutive subarcs of the same length
but with opposite sign 
induces a retraction relative end points
to the rescaled geodesic arc
\[
[0,1]\ni\tau\longmapsto
\pi\circ\gamma_{|u|\sigma_u(s)\cdot\tau}
\left(\frac{u}{|u|}\right)
\]
keeping the oriented length.
The retraction can be defined
by convex interpolation of the parameter functions
$|u|\sigma_u(s\cdot\tau)$ and $|u|\sigma_u(s)\cdot\tau$
and stays inside the trace of the geodesic defined by $u$.
An application of Stokes theorem
to the definition via the action yields the result.
Hence,
\[
\ell_u(s)=|u|\sigma_u(s)
\,.
\]
Moreover,
the proof of the root Lemma \ref{lem:rootlem} shows that
\[
\hat{f}'\big(\varrho(u)\big)=
2\tilde{f}'(r^2)=
\frac{1}{r}f'(r)
\]
with $r=|u|$ for all $u\in T^*Q$ with $|u|\in(0,r_0)$,
so that
\[
\ell_u(s)=
\int_s^0
  f'\big(\rme^{\sigma}|u|\big)
\rmd\sigma
\]
is the desired oriented length $\ell_u(s)$.


\subsection{Characteristic half
length\label{subsec:charhalflength}}

Let $u_0$ in $S_{r_0}T^*Q$ be a boundary point 
and take a sequence $u_{\varepsilon}\ra u_0$
with $|u_{\varepsilon}|=r_0-\varepsilon$
for $\varepsilon>0$ small.
We claim that for all $s<0$ fixed the limit
\[
 \ell_{u_0}(s):=\lim_{\varepsilon\ra0}\ell_{u_{\varepsilon}}(s)
\]
exists and is finite.
Indeed,
by Section \ref{subsec:orientedlength},
$\ell_{u_{\varepsilon}}(s)$ is the oriented length
of the curve
$[0,1]\ni\tau\mapsto\pi\big(\psi_{s\cdot\tau}(u_{\varepsilon})\big)$
and only depends on $|u_{\varepsilon}|$.
The lift
$\big(f(\psi_{s\cdot\tau}(u_{\varepsilon})),\psi_{s\cdot\tau}(u_{\varepsilon})\big)$
to the graph $S$ of $f$
parametrises subarcs of half characteristics,
which intersects the equator for $\varepsilon\ra0$.
By the concluding remark of Section \ref{subsec:fibwradsymm}
the resulting subarcs admit reparametrisations
as integral curves of a characteristic vector field of $S_{\xi_Q}$
with finite time intervals for all $\varepsilon\ra0$.
As the oriented length is invariant under strictly increasing
time reparametrisations the claim follows.

In order to show finiteness towards the singular set
we make the following observations:
With the substitution $\tau=\rme^{\sigma}|u|$,
which yields $\rmd\tau=\tau\rmd\sigma$,
and $r=\rme^s|u|$ the oriented length $\ell_u(s)$, $|u|\in(0,r_0)$, equals
\[
L_u(r):=
\int_r^{|u|}
          \frac{f'(\tau)}{\tau}
\rmd\tau
\]
corresponding to the time reparametrised flow
\[
\Psi_r(u)=r\gamma_{L_u(r)}
\left(\frac{u}{|u|}\right)
\,.
\]
By the Morse lemma from Section \ref{subsec:mlandcons},
the quantity $L_u(r)$, and hence $\Psi_r(u)$,
is also defined for $r$ non-positive,
so that by the previous remarks
\[
L_{u_0}(-r_0)=
\int_{-r_0}^{r_0}
          \frac{f'(\tau)}{\tau}
\rmd\tau
\]
is finite for all boundary points $u_0$ in $S_{r_0}T^*Q$.
The quantity $L_{u_0}(-r_0)$ is the oriented length
of the curve
\[
[-r_0,r_0]\ni r\longmapsto\pi\circ\gamma_{L_{u_0}(r)}
\left(\frac{u_0}{|u_0|}\right)
\,,
\]
which lifts to a {\bf doubled half characteristic}
\[
[-r_0,r_0]\ni r\longmapsto
\big(f(\Psi_r(u_0)),\Psi_r(u_0)\big)
\]
to the graph $S$ of $f$.
At $r=0$ the lift passes through the singular set
and reaches the equator on both ends.

\begin{rem}
Let $(S,V_S,f_{\pm})$ be a shape in $\R\times T^*Q$
with shadow set $V_S$
equal to the interior of $D_{r_0}T^*Q$
and $f_{\pm}$ fibred radially symmetric.
According to the orientation
of characteristic vector fields
in Section \ref{subsec:contactisation}
we have that
\[
L_{u_0}(0)=\pm\LLL(f_{\pm})
\quad\text{and}\quad
L_{u_0}(-r_0)=\pm2\LLL(f_{\pm})
\]
for the oriented length of
a half characteristic and
doubled half characteristic
on the graph of $f_{\pm}$.
\end{rem}


\subsection{Normalised cogeodesic twists\label{subsec:normcogeotwists}}

  The normalised cogeodesic flow $\eta_{\theta}$
  on the deleted cotangent bundle
  $T^*Q\setminus Q$ appears
  as the Hamiltonian flow
  of the length function $l(u):=|u|$, $u\in T^*Q$.
  The generating Hamiltonian vector field satisfies
  $X_l=\frac{1}{l}X_{\varrho}$, so that
  $\gamma_{\theta}=\eta_{l\theta}$.
  Observe that $|X_l|_J=1$ taking the Sasaki metric
  $\rmd\lambda(\,.\,,J\,.\,)$ on $T^*Q$  
  and that $\lambda(X_l)=l$,
  the latter being equivalent to
  $\eta_{\theta}^*\lambda=\lambda$.
  Hence, $\eta_{\theta}$ is a $\lambda$-preserving
  exact symplectomorphism
  on $T^*Q\setminus Q$ for each $\theta\in\R$.
  
  Given any smooth function $\theta\co\R\ra\R$
  we define an exact symplectomorphism
  on $T^*Q\setminus Q$
  by setting
  \[
  \Phi:=\eta_{\theta(l)}
  \,.
  \]
  The inverse equals
  $\Phi^{-1}=\eta_{-\theta(l)}$
  because $\eta_{\theta}$ preserves $l$.
  Computing
  \[
  T\Phi=
  T\eta_{\theta(l)}+
  \big(\theta'(l)\rmd l\big)\cdot X_l|_{\Phi}
  \]
  we obtain
  \[
  \Phi^*\lambda_Q=
  \lambda_Q+\rmd\Theta
  \]
  as $\Phi$ preserves the length of covectors.
  Here, we define
  \[
  \Theta(u):=
  -\int_{|u|}^{r_0}\tau\theta'(\tau)\rmd\tau
  \,,
  \]
  which is fibred radially symmetric.
  
  Choose $\theta\co\R\ra\R$ to be equal to $0$ near $0$
  and constantly equal to $\theta_0=\theta(r_0)$
  near $[r_0,\infty)$.
  Hence,
  $\Theta$ is constant near $0$
  and has support in $[0,r_0)$.
  As $\eta_0=\id$ we see that $\Phi$
  extends over the zero section.
  Furthermore
  $\Phi$ equals $\eta_{\theta_0}$
  in a neighbourhood of $S_{r_0}T^*Q$.
  Define the
  {\bf normalised cogeodesic twist}
  by
  \[
  \varphi(b,u):=\big(b-\Theta(u),\Phi(u)\big)
  \]
  for all $(b,u)\in\R\times T^*Q$.
  Then $\varphi$ is a strict contactomorphism
  preserving the contact form $\rmd b+\lambda_Q$
  and is equal to
  $(b,u)\mapsto\big(b,\eta_{\theta_0}(u)\big)$
  for $|u|$ contained in a neighbourhood of $[r_0,\infty)$.
  As $\eta_{\theta_0}$ preserves $l$,
  the strict contactomorphism $\varphi$
  sends any fibred radially symmetric shape
  $(S,V_S,f_{\pm})$ with $V_S=D_{r_0}T^*Q$
  to the equivalent shape $(T,V_T,g_{\pm})$
  preserving the equator,
  where $g_{\pm}=f_{\pm}-\Theta$
  are fibred radially symmetric.
  Computing the difference yields
  \[
  \LLL(f_{\pm})-\LLL\big(f_{\pm}-\Theta\big)
  =\pm\int_0^{r_0}\theta'(\tau)\rmd\tau
  =\pm\theta_0
  \,.
  \]
In particular,
the normalised cogeodesic twist
preserves the characteristic length
\[
\Sigma(T)=
\LLL\big(f_+-\Theta\big)+
\LLL\big(f_--\Theta\big)=
\Sigma(S)
\]
for all times $\theta_0$.

\begin{rem}
\label{rem:bessetwist}
  If $Q$ is a Besse manifold,
  so that $\eta_{\theta}$ is periodic
  with minimal period $2P$,
  then $\varphi$ restricts to
  $\id\times\,\eta_{kP}$ along a neighbourhood
  of the equator taking $\theta_0=kP$
  for any integer $k$.
  In this case we get
  $\LLL(f_{\pm})-\LLL\big(f_{\pm}-\Theta\big)=\pm kP$.
  
  If $Q=S^n$ is the unit sphere in $\R^{n+1}$
  the flow $\eta_{\theta}$ is periodic
  with minimal period $2\pi$.
  The map $\eta_{\pi}$
  extends over the zero section
  being equal to the antipodal map
  $(\bfq,\bfp)\mapsto(-\bfq,-\bfp)$
  written in coordinates,
  cf.\ \cite[Section 4]{bgz19} or \cite[Section 2.2]{vkn05}.
  Choosing $\theta(0)=j\pi$ to be an appropriate
  multiple of $\pi$ yields the so-called
  {\bf Dehn-Seidel twists} with
  $\LLL(f_{\pm})-\LLL\big(f_{\pm}-\Theta\big)=\pm(k-j)\pi$.
\end{rem}

\begin{proof}[{\bf Proof of Theorem \ref{thmintr:cobocotangentconstr}}]
After an application of an
normalised cogeodesic twist
as in Section \ref{subsec:normcogeotwists}
with time $\theta_0=\LLL(f_+)-\LLL(g_+)$,
which preserves $\Sigma(S)$,
we can assume that
$\LLL(f_{\pm})=\LLL(g_{\pm})$, resp.,
keeping the notation.
The argumentation now goes along the lines
of the proof of the {\it if part} of
Theorem \ref{thmintr:contactbodies}
applied to the situation of
Sections \ref{subsec:tamegeom}
and \ref{subsec:fibwradsymm}.

For the in the interior of $D_{r_0}T^*Q$
compactly supported radially symmetric function
$g:=g_{\pm}-f_{\pm}$ we
identify $g(u)=g(r)$, where $r=|u|$,
so that by the root Lemma \ref{lem:rootlem}
$g=\hat{g}(\varrho)$ factors through a smooth function
$\hat{g}=\tilde{g}\circ2\co\R\ra\R$
as $r\mapsto g(r)$ is even.
By Section \ref{subsec:symmsteinshadow}
we are free to work with the flow
$\psi_s(u)=\rme^su$ of
the autonomous radial Liouville vector field
$Y_{\lambda}$.

For $u_0\in S_{r_0}T^*Q$
consider the flow line
$c(s)=\rme^su_0$ for all $s<0$.
For all $u\in T^*Q$ with $|u|\in(0,r_0)$
write $u=\rme^su_0$ and define $h(u)=h(r)$ by
  \[
    h(u)=
    \rme^s
        \int_s^0
          \rme^{-\sigma}
          \Big(g\big(r_0\rme^{\sigma}\big)\Big)'
          \,
        \rmd\sigma
        \,.
  \]
  By chain rule the integrand equals
  $g'\big(r_0\rme^{\sigma}\big)$,
  so that the substitution $\tau=r_0\rme^{\sigma}$,
  which gives $\rmd\tau=\tau\rmd\sigma$
  and $r=r_0\rme^s$, yields
  \[
    h(u)=r
        \int_r^{r_0}
          \frac{g'(\tau)}{\tau}
          \,
        \rmd\tau=
        -r
        \int_0^r
          \frac{g'(\tau)}{\tau}
          \,
        \rmd\tau
        \,.
  \]
  The last equality follows from
  $\LLL(g_{\pm})=\LLL(f_{\pm})$, resp.
  
  It remains to show smoothness near $\{u=0\}$.
  By the Morse lemma in Section \ref{subsec:mlandcons}
  the integrant is equal to $(g')_1(\tau)$,
  which is an even smooth function on $\R$.
  Hence,
  the function
  $h(r):=-r\int_0^r(g')_1(\tau)\rmd\tau$
  is even and smooth on $\R$,
  so that the root Lemma \ref{lem:rootlem} yields
  a smooth function
  $\tilde{h}$ on $\R$ with $h(r)=\tilde{h}(r^2)$.
  Therefore,
  $h(u)=\tilde{h}(|u|^2)$
  is smooth near $\{u=0\}$ also.
  The claim follows with 
  Proposition \ref{prop:charisteqn}
  and Section \ref{subsec:characteristics}
  as in the proof of Theorem \ref{thm:similar}.
  
  That the resulting contactomorphism
  restricts to a $b$-shift of the singular sets
  follows with
  Lemma \ref{lem:verticallydiffeomorphic} (ii)
  and
  Section \ref{subsec:normcogeotwists}.
  Having computed $h$
  the integrating vector field in  
  Remark \ref{rem:singsetshift}
  turns out to be Hamiltonian
  proportional to $X_{\varrho}$,
  because the Hamiltonian functions involved
  are fibred radially symmetric.
\end{proof}


\subsection{An equator preserving contact invariant\label{subsec:equatorpresvinv}}

As on the end of Section \ref{subsec:charhalflength}
consider a shape $(S,V_S,f_{\pm})$ in $\R\times T^*Q$
with shadow set $V_S$
equal to the interior of $D_{r_0}T^*Q$
and $f_{\pm}$ fibred radially symmetric.
For all boundary points $u_0$ in $S_{r_0}T^*Q$
abbreviate
\[
L^{\pm}(r):=
\int_r^{r_0}
          \frac{f_{\pm}'(\tau)}{\tau}
\rmd\tau
\]
for all $r\in[-r_0,r_0]$.
This yields curves
\[
u^{\pm}(r)=
\Psi_r(u_0)=
r\gamma_{L^{\pm}(r)}
\left(\frac{u_0}{r_0}\right)
\]
in $T^*Q$,
which lift to a parametrisation of
doubled half characteristics
\[
[-r_0,r_0]\ni r\longmapsto
\big(f_{\pm}(u^{\pm}(r)),u^{\pm}(r)\big)
\]
to the graph $S^{\pm}$ of $f_{\pm}$.
Each lift passes through the singular set for $r=0$
and intersects the equator
\[
E:=\{f_+(r_0)\}\times S_{r_0}T^*Q
\]
for $r=\pm r_0$.
The assignment of $u_0$ to
\[
u^{\pm}(-r_0)=
-r_0\gamma_{\sigma_{\pm}}
\!\left(\frac{u_0}{r_0}\right)
\,,
\quad
\sigma_{\pm}:={\pm}2\LLL(f_{\pm})
\,,
\]
defines an involution
\[
\mu_{S^{\pm}}:=
-\gamma_{\sigma_{\pm}/r_0}
\]
on the equator $E$ such that
\[
\mu_{S^{\pm}}=
\gamma_{-\sigma_{\pm}/r_0}\circ (-1)
\,.
\]
This follows from
$x\gamma_{\sigma}\circ\frac1x=\gamma_{\sigma/x}$
for all $x\in\R$ non-zero,
see Section \ref{subsec:halfchardyn}.
The case $x=-1$ is directly implied by the definition of geodesics.
The composition
\[
\mu_S:=\mu_{S^-}\circ\mu_{S^+}
\,,
\]
which assigns $u^+(-r_0)$ to $u^-(-r_0)$,
is the strict contactomorphism
\[
\mu_S=\gamma_{\Sigma/r_0}
\,,
\qquad
\Sigma:=2\Sigma(S)
\,,
\]
of $(E,\alpha^{r_0})$.

We remark that
$\gamma_{\sigma/r_0^2}$ is the Reeb flow
at time $\sigma/r_0^2$
on $S_{r_0}T^*Q$ of the contact form $\alpha^{r_0}$
obtained from restricting $\lambda_Q$
to $TS_{r_0}T^*Q$.
Denoting the Reeb vector field of $\alpha^{r_0}$
by $R_{\alpha^{r_0}}$ this implies that
$\gamma_{\sigma/r_0}$ is the flow of
$\frac{1}{r_0}X_{\varrho}=r_0R_{\alpha^{r_0}}$
on $S_{r_0}T^*Q$.
Therefore,
any diffeomorphism $\varepsilon$ of $E$
such that
  	\begin{enumerate}
 	\item[($\varepsilon_{-1}$)]
	$\varepsilon$ commutes with $(-1)$ and
 	\item[($\varepsilon_{\alpha^{r_0}}$)]
	$\varepsilon$ preserves $\alpha^{r_0}$
	\end{enumerate}
satisfies the following rule:
Conjugation of $\gamma_{\sigma/r_0}$ with
$\varepsilon\circ (-1)$ is the flow of the push forward
\[
\big(\varepsilon\circ (-1)\big)_*
\Big(\frac{1}{r_0}X_{\varrho}\Big)
=-\frac{1}{r_0}X_{\varrho}|_{(-1)}
\,,
\]
which is $-\gamma_{\sigma/r_0}\circ (-1)$.
In other words,
\[
\varepsilon\circ\mu_{S^{\pm}}=
\mu_{S^{\pm}}\circ\varepsilon
\,.
\]
Similarly,
with the only need of using ($\varepsilon_{\alpha^{r_0}}$),
we get
\[
\varepsilon\circ\mu_S=
\mu_S\circ\varepsilon
\,.
\]

\begin{proof}[{\bf Proof of Theorem \ref{thmintr:cobocotangentobstr}}]
  Let $\varphi$ be an co-orientation
  preserving contactomorphism
  from $(D_S,\xi_{\st})$ onto $(D_T,\xi_{\st})$
  that restricts to a diffeomorphism
  $\varepsilon\co E\ra E$ of the equator
  $\{f_+(r_0)\}\times S_{r_0}T^*Q$.
  By Proposition \ref{prop:phiinvonchrfol},
  $\varphi$ maps the signed singular set
  $\pm\{f_{\pm}(0)\}\times Q$ to
  $\pm\{g_{\pm}(0)\}\times Q$,
  so that $\varphi(S^{\pm})=T^{\pm}$
  and, as $\varphi$ preserves the equator,
  $\varphi$ sends the doubled half characteristics of
  $(S^{\pm})_{\xi_Q}$ to the one of $(T^{\pm})_{\xi_Q}$.
  In view of the end points we get
  \[
  \varepsilon\circ\mu_{S^{\pm}}=
  \mu_{T^{\pm}}\circ\varepsilon
  \]
  and, as $\mu_{S^{\pm}}$ and $\mu_{T^{\pm}}$ are involutive,
  \[
  \varepsilon\circ\mu_S=
  \mu_T\circ\varepsilon
  \,.
  \]
  Under the assumptions
  ($\varepsilon_{-1}$), ($\varepsilon_{\alpha^{r_0}}$)
  and ($\varepsilon_{\alpha^{r_0}}$)
  as above we get
  $\mu_{S^{\pm}}=\mu_{T^{\pm}}$ and
  $\mu_S=\mu_T$, resp.,
  applying the commutation rules.
  Therefore,
  we find integers $k_{\pm}$, $k$ such that
  \[
  \LLL(f_{\pm})=\LLL(g_{\pm})+k_{\pm}P
  \qquad
  \text{and}
  \qquad
  \Sigma(S)=\Sigma(T)+kP
  \,,
  \]
  resp., where $2P/r_0$ is the minimal period
  of $\gamma_{\sigma}$ on $S_{r_0}T^*Q$,
  in case $Q$ is Besse,
  see Wadsley \cite{wad75};
  reading $P=0$ for $Q$ non-Besse.
  This proves (i) and (ii).
  
In view of (iii) we define
a strict contactomorphism
\[
\kappa_{S^{\pm}}:=
\gamma_{\tau_{\pm}/r_0}
\,,
\quad
\tau_{\pm}:={\pm}\LLL(f_{\pm})
\,,
\]
on the equator $(E,\alpha^{r_0})$,
which corresponds to the assignment of
\[
\kappa_{S^{\pm}}(u_0)=
r_0\gamma_{\tau_{\pm}}
\!\left(\frac{u_0}{r_0}\right)
\]
to $u_0$ similarly to the above considerations.
Observe that
$\pi\circ\kappa_{S^{\pm}}(u_0)=u^{\pm}(0)$
and that $\big(f_{\pm}(0),u^{\pm}(0)\big)$
is a point of the singular set $\{f_{\pm}(0)\}\times Q$,
which gets mapped onto $\{g_{\pm}(0)\}\times Q$ by
$\varphi|_{\{f_{\pm}(0)\}\times Q}=\big(g_{\pm}(0),\iota\big)$.
Preservation of characteristics yields
\[
\varphi\circ
\big(f_{\pm}(0),\pi\circ\kappa_{S^{\pm}}\big)
=
\big(g_{\pm}(0),\pi\circ\kappa_{T^{\pm}}\big)
\circ\varphi
\]
restricted to the equator $E$,
which is
\[
\Big(g_{\pm}(0),\iota\circ\pi\circ\kappa_{S^{\pm}}\Big)
=
\Big(g_{\pm}(0),\pi\circ\kappa_{T^{\pm}}\circ\varepsilon\Big)
\,.
\]
Using the assumption $\iota\circ\pi=\pi\circ\varepsilon$
and that $\varphi|_E=\varepsilon$ commutes with
$\gamma_{\tau_{\pm}/r_0}$ by ($\varepsilon_{\alpha^{r_0}}$)
we get
$\pi\circ\kappa_{S^{\pm}}=\pi\circ\kappa_{T^{\pm}}$.
This implies that
\[
\pi\circ\gamma_{\pm(\LLL(f_{\pm})-\LLL(g_{\pm}))/r_0}
=\pi
\]
on $E$ meaning that
$\gamma_{\pm(\LLL(f_{\pm})-\LLL(g_{\pm}))/r_0}$
preserves the fibres of $\pi\co E\ra Q$.
If $\LLL(f_{\pm})-\LLL(g_{\pm})$ is not zero
then all points $q$ in $Q$ have the property,
that all geodesics starting at $q$ come back to $q$
after time $\pm(\LLL(f_{\pm})-\LLL(g_{\pm}))/r_0$.
By \cite[Proposition 7.9(b)]{bes78},
which is based on Wadsley's theorem \cite{wad75},
the flow $\gamma_{\tau}$ on $E$ is periodic
with minimal period $2P/r_0$, say.
Observe,
that transverse self-intersection points
of a given geodesic $q=q(\tau)$ in $Q$ are finite.
Indeed,
the set of $(\tau_1,\tau_2)$ in $\R\times\R$
such that $q(\tau_1)=q(\tau_2)$
but $\dot q(\tau_1)\neq\dot q(\tau_2)$ are isolated
using metric balls of radius equal to the convexity radius of $Q$
about transverse self-intersections of $q=q(\tau)$,
cf.\ \cite[Lemma 7.10]{bes78}.
Taking a point $q$ on a geodesic of $Q$
with minimal period $2P/r_0$
that is not a transverse self-intersection point
shows that there exists integers $k_{\pm}$ with
  \[
  \LLL(f_{\pm})=\LLL(g_{\pm})+2k_{\pm}P
  \,.
  \]
This shows (iii).
\end{proof}


\section{Shadows in symplectisations\label{sec:shadowsympl}}

We address Theorem \ref{thmintr:contactbodiessympl}.


\subsection{Stein type symplectisations}
\label{subsec:steintypesympl}

Recall the situation from the beginning of
Section \ref{subsec:symmsteinshadow}.
Assume
that the strict closed contact manifold $(M,\alpha)$
considered in
Theorem \ref{thmintr:contactbodiessympl}
appears as boundary
of $\{\varrho<\varrho_*\}\subset V$
with contact form $\alpha$ being the restriction
of the Liouville primitive $\lambda=-\rmd\varrho\circ J$
to the tangent spaces of $\{\varrho=\varrho_*\}$.
Then a Liouville symplectic embedding
$\Psi(a,w):=\psi_a(w)$
of $(\R\times M,\rme^a\alpha)$
into $(V,\lambda)$
is defined by following the Liouville flow
$\psi_a$ of $Y_{\lambda}$.
Notice that $\Psi^*Y_{\lambda}=\partial_a$
is the $\rmd(\rme^a\alpha)$-dual Liouville vector field.

Assume further
that $\Psi^*\varrho\equiv\varrho(a)$ is an
$\R$-symmetric function on the symplectisation 
$(\R\times M,\rme^a\alpha)$
as it is the case in \cite[Section 5.5]{kwz22}
and (taking $\varrho_*=1$)
  	\begin{enumerate}
 	\item[(r2n)]
	for Euclidean spaces as in the {\it if part}
	of proof of Theorem \ref{thmintr:contactbodies}
	in Section \ref{subsec:symmsteinshadow},
	where $\varrho(a)=\rme^a$, or
 	\item[(t*q)]
	for cotangent bundles as described in
	Section \ref{subsec:tamegeom},
	where $\varrho(a)=\rme^{2a}$.
	\end{enumerate}
Writing $I=\Psi^*J$ we get
\[
\rme^a\alpha=-\varrho'(a)\rmd a\circ I
\,,
\]
with $\varrho'(a)$ necessarily positive,
and, therefore,
\[
I\partial_a
=X_{\varrho(a)}
=\tfrac{\varrho'(a)}{\rme^a}R_{\alpha}
\]
denoting the Reeb vector field of $\alpha$
by $R_{\alpha}$.
This yields
\[
\rme^a\alpha(X_{\varrho(a)})
=\partial_a(\varrho(a))
=\varrho'(a)
\]
and
\[
[\partial_a,X_{\varrho(a)}]
=\Big(
\tfrac{\varrho''(a)}{\varrho'(a)}-1
\Big)X_{\varrho(a)}
\,.
\]
The $\Psi$-relatum of both formulas
was used in the (t*q)-case
in Section \ref{subsec:halfchardyn}, resp.

Consider the induced strict contact embedding
$\id\times\,\Psi$
of the related contactisations.
Let $(S,V_S,f_{\pm})$ be a shape
contained in the image
such that $f_{\pm}=\hat{f}_{\pm}(\varrho)$
with $V_S=\{\varrho_1<\varrho<\varrho_0\}$
for real numbers $\varrho_1,\varrho_0$.
The preimage $(\id\times\,\Psi)^{-1}(S)$
defines a shape with shadow set
$\Psi^{-1}(V_S)=\{a_1<a<a_0\}$
for $a_{1/0}=\varrho^{-1}(\varrho_{1/0})$
and $\R$-symmetric
$\Psi^*f_{\pm}\equiv\hat{f}_{\pm}\big(\varrho(a)\big)$.
The function $\Psi^*h\equiv\hat{h}\big(\varrho(a)\big)$
related to the solution of the Moser problem from
Section \ref{subsec:symmsteinshadow} becomes
  \[
    \hat{h}\big(\varrho(s+a_0)\big)=
    \rme^s
        \int_s^0
          \rme^{-\sigma}
          \Big(
          	(\hat{g}_{\pm}-\hat{f}_{\pm})
		\big(\varrho(\sigma+a_0)\big)\Big)'\,
        \rmd\sigma
  \]
for $s\in(a_1-a_0,0)$,
which of course resamples
earlier computations done in the cases
(r2n) and (t*q).


\subsection{Action and characteristics}
\label{subsec:acandchar}

We consider $\R\times M$
equipped with $\rme^a\alpha$
and define a (non-proper) $\R$-symmetric potential
$\varrho(a)=\rme^a$.
Observe that
$-\rmd\varrho\circ I=\rme^a\alpha$
or, equivalently, that $\alpha=-\rmd a\circ I$,
for a translation invariant almost complex structure
$I$ on $\R\times M$ that restricts to a
$\rmd\alpha$-compatible complex structure
on $\ker\alpha$
such that $I\partial_a=R_{\alpha}$.

Let $(S,V_S,f_{\pm})$ be a shape
in $\big(\R\times(\R\times M),\rmd b+\rme^a\alpha\big)$
with $V_S=(a_1,a_0)\times M$
and $f=f_{\pm}$ being $\R$-symmetric.
Compute the Hamiltonian
\[
X_f=f'(a)X_a=\rme^{-a}f'(a)R_{\alpha}
\]
using chain rule.
Hence,
the flow of the nowhere singular vector field
$\partial_a-X_f$ equals
\[
\psi_s(a,w)=
\big(
s+a,
\eta_{\sigma_a(s)}(w)
\big)
\,,
\]
where $\eta_{\sigma}$
denotes the Reeb flow of $\alpha$
and
\[
\sigma_a(s)=\int_{s+a}^a
        \rme^{-\tau}f'(\tau)\,
        \rmd\tau
\,.
\]
The flow of $\pm(\partial_a-X_{f_{\pm}})$
lifts to a paramatrisation of characteristics
on $S^{\pm}$.
As the contact structure $\xi_{\alpha}$
on $\R\times(\R\times M)$
is nowhere horizontal
all characteristics intersect the equator
$E=\{a_1,a_0\}\times M$ transversally
connecting the two components of $E$
in finite time.
If follows
that the two sided limit of
\[
\AAA(f_{\pm})=
\pm\int_{a_1}^{a_0}
\rme^{-\tau}
f'_{\pm}(\tau)\,
\rmd\tau
\]
exists and is finite,
cf.\ Section \ref{subsec:charhalflength}.
The quantity $\AAA(f_{\pm})$
is the $\alpha$-action
of a Reeb arc in $M$
obtained by projecting
an equator connecting characteristic on $S^{\pm}$
intersecting $\{a_1\}\times M$ and $\{a_0\}\times M$
exactly once.

\begin{proof}
[{\bf Proof of Theorem \ref{thmintr:contactbodiessympl} (a)}]
Consider the $\R$-symmetric function
$g:=g_{\pm}-f_{\pm}$ defined on $(a_1,a_0)\times M$.
Observe that the function $g=g(a)$
has support in $(a_1,a_0)$.
As argued in Section \ref{subsec:symmsteinshadow},
we work with the flow
$\psi_s(a,w)=(s+a,w)$ of $\partial_a$.
For an initial point $(a_0,w)\in\{a_0\}\times M$
consider the flow line
$c(s)=(s+a_0,w)$ for all $s\in(a_1-a_0,0)$.
For the points $(a,w)\in(a_1,a_0)\times M$
write $a=s+a_0$ and define
$h_{\pm}(a,w)=h(a)$ by
  \[
    h(a)=
    \rme^s
        \int_s^0
          \rme^{-\sigma}
          \big(
          	g(\sigma+a_0)\big)'\,
        \rmd\sigma
        \,.
  \]
The substitution $\tau=\sigma+a_0$,
which gives $a=s+a_0$, leads to
  \[
    h(a)=
    \rme^a
        \int_a^{a_0}
        \rme^{-\tau}g'(\tau)\,
        \rmd\tau
        \,.
  \]
  In order to ensure smoothness 
  of the desired Hamiltonian of the Moser porblem
  as in Proposition \ref{prop:charisteqn}
  caused by the branches $h_+$ and $h_-$
  near the lower boundary $\{a_1\}\times M$
  we need to require $h_+(a_1)=h_-(a_1)$
  or, equivalently,
  \[
        \int_{a_1}^{a_0}
        \rme^{-\tau}
        \big(g_+-f_+\big)'(\tau)\,
        \rmd\tau
        =
        \int_{a_1}^{a_0}
        \rme^{-\tau}
        \big(g_--f_-\big)'(\tau)\,
        \rmd\tau
  \]
  ensured by the assumption
  $\TT(S)=\TT(T)$
  on the total characteristic actions.
  As in the proof of Theorem \ref{thm:similar}
  the claim follows with 
  Proposition \ref{prop:charisteqn}
  and Section \ref{subsec:characteristics}.
  The statement about the support
  follows using Remark \ref{rem:singsetshift}.
\end{proof}


\subsection{Reeb twists\label{subsec:reebtwists}}

As in Section \ref{subsec:normcogeotwists}
consider the Reeb flow $\eta_{\theta}$ of $(M,\alpha)$,
which satisfies
$\eta_{\theta}^*\alpha=\alpha$.
Take a smooth function $\theta\co\R\ra\R$
such that $\theta$ is constant
near $(-\infty,a_1]\cup[a_0,\infty)$
with $\theta(a)_{1/0}=\theta_{1/0}$, say,
and $\Theta(a_1)=0$, where
  \[
  \Theta(a):=
  -\int_{a}^{a_0}\rme^{\tau}\theta'(\tau)\rmd\tau
  \]
having support in $(a_1,a_0)$.
Define an exact symplectomorphism
  \[
  \Phi:=\id\times\,\eta_{\theta}
  \]
on $\R\times M$.
Indeed, the inverse is given by
$\Phi^{-1}=\id\times\,\eta_{-\theta}$.
The linearisation is computed to be
  \[
  T\Phi=
  \mathbb{I}\oplus
  \Big(
  T\eta_{\theta}+
  \rmd\theta\cdot R_{\alpha}|_{\Phi}
  \Big)
  \,,
  \]
  so that
  \[
  \Phi^*(\rme^a\alpha)=
  \rme^a\alpha+\rmd\Theta
  \,.
  \]
Define the {\bf Reeb twist} by
  \[
  \varphi(b,a,w)
  :=\big(b-\Theta(a),a,\eta_{\theta(a)}(w)\big)
  \]
for all $(b,a,w)\in\R\times\R\times M$.
Then $\varphi$ preserves the contact form
$\rmd b+\rme^a\alpha$, equals
$(b,a,w)\mapsto\big(b,a,\eta_{\theta_{1/0}}(w)\big)$
for $a$ close to $a_{1/0}$ and
sends any $\R$-symmetric shape
$(S,V_S,f_{\pm})$ with $V_S=(a_1,a_0)\times M$
to the equivalent $\R$-symmetric shape $(T,V_T,g_{\pm})$
preserving the equator $\{a_1,a_0\}\times M$,
where $g_{\pm}=f_{\pm}-\Theta$.
The characteristic action difference is
  \[
  \AAA(f_{\pm})-\AAA\big(f_{\pm}-\Theta\big)
  =\pm\int_{a_1}^{a_0}\theta'(\tau)\rmd\tau
  =\pm(\theta_0-\theta_1)
  \,,
  \]
so that the Reeb twist
preserves the total characteristic action
\[
\TT(T)=
\AAA\big(f_+-\Theta\big)+
\AAA\big(f_--\Theta\big)=
\TT(S)
\,.
\]


\subsection{Equator monodromy}\label{subsec:eqmono}

We assume the situation of Section \ref{subsec:acandchar}.
As in Section \ref{subsec:equatorpresvinv}
we consider curves in $(a_1,a_0)\times M$
whose lift to $S^{\pm}$ parametrises a characteristic
that starts and ends at one of the two components
of the equator $\{a_1,a_0\}\times M$ and
intersects the second component exactly once.
The projection of the end points and the starting points to $M$,
going backwards thought the characteristic arcs as described
in time, yields a map
\[
\kappa_S:=
\eta_{\TT(S)}
\,.
\]
The map $\kappa_S$
is a strict contactomorphism
of $(M,\alpha)$ given by the time $\TT(S)$-map
of the Reeb flow $\eta_{\theta}$.
The intersections of the characteristic arcs
with the closure of $S^{\pm}$
define strict contactomorphisms
\[
\kappa_{S^{\pm}}:=
\eta_{\AAA(f_{\pm})}
\]
of $(M,\alpha)$.
The maps $\kappa_S, \kappa_{S^{\pm}}$
commute with strict contactomorphisms
$\varepsilon$ of $(M,\alpha)$.

\begin{proof}[{\bf Proof of Theorem
\ref{thmintr:contactbodiessympl} (b)}]
  Starting with (i)
  let $\varphi$ be an co-orientation
  preserving contactomorphism
  from $(D_S,\xi_{\alpha})$ onto $(D_T,\xi_{\alpha})$
  that restricts to a strict contactomorphism
  $\varepsilon_{1/0}$ of the component
  $\{a_{1/0}\}\times M$ of the equator
  preserving the co-orientation.
  By Proposition \ref{prop:phiinvonchrfol},
  $\varphi$ sends the oriented characteristics of
  $S_{\xi_{\alpha}}$ to the one of
  $T_{\xi_{\alpha}}$ so that
  $
  \varepsilon_{1/0}\circ\kappa_S=
  \kappa_T\circ\varepsilon_{1/0}
  $.
  Commutation yields $\kappa_S=\kappa_T$,
  which shows (i).
  
  In order to prove (ii)
  observe that $\varphi$ restricts to
  a co-orientation preserving strict contactomorphism
  on the equator $\{a_1,a_0\}\times M$
  not interchanging its components.
  This implies $\varphi(S^{\pm})=T^{\pm}$. 
  Moreover,
  $\varphi$ restricts to the same map
  $\varepsilon=\varepsilon_{1/0}$
  on each component of $\{a_{1/0}\}\times M$
  identified with $M$.
  In view of the end points of the
  characteristics of $(S^{\pm})_{\xi_{\alpha}}$
  and their $\varphi$-images $(T^{\pm})_{\xi_{\alpha}}$
  we see that
  $\varepsilon\circ\kappa_{S^{\pm}}=
  \kappa_{T^{\pm}}\circ\varepsilon$,
  see Proposition \ref{prop:phiinvonchrfol}.
  As $\varepsilon$ preserves
  the Reeb vector field of $(M,\alpha)$
  we get $\kappa_{S^{\pm}}=\kappa_{T^{\pm}}$, resp.,
  by commutation.  
  This proves (ii).
\end{proof}

\begin{rem}
\label{rem:nongirouxconvex}
  The inverse of the strict contactomorphism $\kappa_S$
  allows an identification of $S$ with the mapping torus
  $S(\kappa_S^{-1})$.
  Indeed,
  following the $M$-invarinat characteristic flow lines
  starting on a component of the equator
  yields a diffeomorphism onto the quotient
  of $\R\times M$ by the action
  $(t,m)\mapsto\big(t-1,\kappa_S^{-1}(m)\big)$.
  Characteristics will correspond to the curves
  $t\mapsto[(t,m)]$.
  In particular,
  $S$ cannot be Giroux-convex with respect to $\xi_{\alpha}$.
  
  Otherwise,
  we would find a smooth function on $S$,
  whose non-empty zero set $\Gamma$ is positively
  transverse to $S_{\xi_{\alpha}}$,
  see Remark \ref{rem:divset}.
  As $\{0\}\times M$ does not separate in $S(\kappa_S^{-1})$
  by an open-closed argument $\Gamma$ cannot be contained
  in $\{0\}\times M$.
  Hence,
  the set $A_k=(\kappa_S^k)(A_0)$ of all $m\in M$
  such that there exists $t\in(k,k+1)$ with $(t,m)\in\Gamma$
  for given integer $k$ is non-empty and open.
  As each characteristic intersects $\Gamma$
  at most once the $A_k$ are pairwise disjoint.
  In addition,
  the $A_k$ are of equal $\alpha$-contact volume.
  This is not possible.
  \end{rem}


\begin{ack}
  The content of this work is part of a lecture course
  given in 2024/25 at RUB
  and was caused by discussion in the A5/C5 Seminar.
  We would like to thank all the members
  Johanna Bimmermann,
  Florian Buck,
  Lars Kelling,
  Christopher Schmidt,
  Bernd Stratmann,
  Manuel Stange
  and Anton Wilke
  for their contributions.
  We thank Alberto Abbondandolo,
  Barney Bramham,
  Boto von Querenburg and Stefan Suhr.
\end{ack}


\end{document}